\documentclass[draftcls,onecolumn,twoside,american]{IEEEtran}
\usepackage[T1]{fontenc}
\usepackage[latin1]{inputenc}
\usepackage{babel} 
\usepackage{amssymb}
\usepackage{amsmath}
\usepackage{amsfonts}
\makeatletter
\let\NAT@parse\undefined
\makeatother
\usepackage{natbib}
\usepackage{algorithm}
\usepackage{algorithmic}
\makeatletter

\providecommand{\LyX}{L\kern-.1667em\lower.25em\hbox{Y}\kern-.125emX\@}

\newtheorem{thm}{Theorem} 
\newtheorem{cor}{Corollary} 
\newtheorem{lem}{Lemma} 
\newtheorem{prop}{Proposition}

\newtheorem{dfn}{Definition}
\newtheorem{rem}{Remark}
\usepackage{verbatim}
\usepackage{setspace}

\usepackage{bbm}
\usepackage{amsbsy}
\usepackage{babel}
\usepackage{color}
\newcommand{\mathd}{\mathrm{d}}

\newcommand{\mathe}{\mathrm{e}}

\definecolor{grey}{rgb}{0.75,0.75,0.75}
\definecolor{orange}{rgb}{1.0,0.5,0.5}
\definecolor{brown}{rgb}{0.5,0.25,0.0}
\definecolor{pink}{rgb}{1.0,0.5,0.5}

\bibpunct{(}{)}{,}{a}{,}{;}
\makeatother

\begin{document}
{ \title{{ 
Coding on countably infinite alphabets}}}
\author{St\'ephane Boucheron, Aurélien Garivier and Elisabeth Gassiat}
\date{\today}
\maketitle
\newcommand{\N}{{\mathbb{N}}}
\newcommand{\R}{{\mathbb{R}}}
\newcommand{\E}{{\mathbb{E}}}
\newcommand{\PROB}{{{P}}}
\newcommand{\nml}{\text{{\sc nml}}}       
\newcommand{\Np}{\N_+}
\newcommand{\equivaut}{\Leftrightarrow}
\newcommand{\nmlnl}{\nml_n(\Lambda)}
\newcommand{\M}{{\mathfrak M}}
\newcommand{\EXP}{\mathbbm{E}}
\newcommand{\longrightarrowlim}{\mathop{\longrightarrow}\limits}
\def\1{\mathbbm{1}}

\begin{abstract}
This paper describes universal lossless coding strategies for compressing
sources on countably infinite alphabets. Classes of memoryless
sources defined by an envelope condition on the marginal distribution
provide benchmarks for coding techniques originating from the theory
of universal coding over finite alphabets. We prove general
upper-bounds on minimax regret and lower-bounds on minimax redundancy
for such source classes. The general upper bounds emphasize the role of the 
Normalized Maximum Likelihood codes with respect to minimax regret 
in the infinite alphabet context.  Lower bounds are derived by
tailoring sharp bounds on the redundancy of Krichevsky-Trofimov coders
for sources over finite alphabets. 
Up to logarithmic (resp. constant) factors the bounds are matching for
source classes defined by algebraically declining (resp. exponentially 
vanishing) envelopes. Effective and (almost) adaptive coding
techniques are described for the collection of source classes defined 
by algebraically vanishing envelopes. Those results extend our
knowledge concerning universal coding to contexts where the key tools
from 
parametric inference 
are known to fail. 
\end{abstract}

{\bf{keywords}:}
\textsc{nml}; countable alphabets; redundancy; adaptive compression;
minimax;

\section{Introduction}
\label{sec:intro}


This paper is concerned with  the problem of universal coding on a countably
infinite alphabet $\mathcal{X}$ 
(say the set of positive integers $\Np$ or
the set of integers $\N$ ) 
as described for example by
\cite{MR2097043}. 
Throughout this paper, a source on the countable alphabet $\mathcal{X}$ 
is a probability distribution on the set $\mathcal{X}^\N$ of infinite sequences of symbols from $\mathcal{X}$
 (this set is endowed with the $\sigma$-algebra generated by sets of the form $\prod_{i=1}^n\{x_i\}\times \mathcal{X}^\N$ where all $x_i\in \mathcal{X}$ and $n\in \N$).
The symbol $\Lambda$ will be used to  denote various classes of sources on the
countably infinite alphabet $\mathcal{X}.$ 
The sequence of symbols emitted by a source is denoted by the ${\cal
  X}^{\mathbbm{N}}$-valued random
variable $\mathbf{X}=\left(X_n\right)_{n\in\N}. $ 
If $\PROB$  denotes the distribution of $\mathbf{X},$ $\PROB^n$ denotes the
distribution of $X_{1:n}=X_1,...,X_n,$ and we 
let $\Lambda^n=\{\PROB^n : \PROB \in\Lambda\}$. 
For any countable set $\mathcal{X}$, let $\M_1(\mathcal{X})$ be the set of all probability measures
on $\mathcal{X}$.

From Shannon noiseless coding Theorem \citep[see][]{cover:thomas:1991},  the binary
entropy of $\PROB^n,$ 
$H(X_{1:n})=\E_{\PROB^n}\left[-\log \PROB(X_{1:n})\right]$ provides a
tight lower bound on the expected number of binary symbols needed to
encode outcomes of $\PROB^n$. 
Throughout the paper, logarithms are in base $2$. 
In the following, we shall only consider finite entropy sources on countable alphabets,
and we implicitly assume that $H(X_{1:n})<\infty$.
The \emph{expected redundancy} of any distribution $Q^n\in
 {\mathfrak{M}}_1({\mathcal{X}}^n)$, defined as
 the  difference between the expected code length 
 $\E_\PROB\left[-\log Q^n(X_{1:n})\right]$ and $H(X_{1:n}),$ is equal to
the Kullback-Leibler divergence (or relative entropy)  $D(\PROB^n,
 Q^n)=\sum_{\mathbf{x}\in {\cal X}^n} \PROB^n\{\mathbf{x}\} 
 \log \frac{\PROB^n(\mathbf{x})}{Q^n(\mathbf{x})} =\EXP_{\PROB^n}\left[\log 
 \frac{\PROB^n(X_{1:n})}{Q^n(X_{1:n})}\right]$.


Universal coding attempts to develop sequences of coding probabilities $(Q^n)_n$  so as to
minimize expected redundancy over a whole class of sources. 
Technically speaking,  several distinct notions of universality have been considered in the
literature. A positive, real-valued function $\rho(n)$  is said to be a strong (respectively weak)
\emph{universal redundancy rate} for a class of sources $\Lambda$ if there
exists a sequence of coding probabilities $(Q_n)_n$ such that for all $n,$
$R^+(Q^n, \Lambda^n)=\sup_{P\in \Lambda} D(P^n,Q^n)\leq \rho(n) $ 
(respectively for all $P\in \Lambda,$ there exists a constant $C(P)$ such
that  for all $n,$ $ D(P^n,Q^n)\leq C(P)\rho(n) $). 
A redundancy rate $\rho(n)$
is said to be non-trivial if $\lim_n \frac{1}{n}\rho(n)=0.$
Finally a class $\Lambda$  of sources will be said to be \emph{feebly universal}
if there exists a single sequence of coding probabilities
$(Q^n)_n$  such that $\sup_{\PROB \in \Lambda} \lim_n \frac{1}{n}
D(\PROB^n,Q^n) = 0 \, $ (Note that this notion of feeble universality
is usually called weak universality, \cite[see][]{MR514346,MR1281931}, we deviate from the tradition, in
order to avoid confusion with the notion of weak universal redundancy rate).  

The \emph{maximal redundancy}  of $Q^n$ with respect to $\Lambda$ is
defined by: 
$$R^+(Q^n, \Lambda^n)=\sup_{\PROB\in\Lambda} D(\PROB^n, Q^n) \, . $$ 
The infimum  of $R^+(Q^n,\Lambda^n)$ is called the \emph{minimax redundancy} with respect to $\Lambda$: 
$$R^+(\Lambda^n)=\inf_{Q^n\in {\M}_1\left({\mathcal{X}}^n\right)}
R^+(Q^n,\Lambda^n).$$ It is the smallest strong universal redundancy rate
for $\Lambda.$ When finite,  it is often called the information radius
of~$\Lambda^n.$ 

As far as finite alphabets are concerned, it is well-known  that the
class of stationary ergodic sources is feebly universal. This is witnessed by the performance 
of Lempel-Ziv codes \citep[see][]{cover:thomas:1991}. It is also known that 
the class of stationary ergodic sources over a finite alphabet does
not admit any non-trivial weak universal redundancy rate \citep{shields:1993}.
On the other hand, fairly large  
 classes of sources admitting strong universal redundancy rates  and
 non-trivial weak universal redundancy rates have been exhibited
\citep[see][ and references
therein]{barron:rissanen:yu:1998,catoni:2004}.
In this paper, we will mostly focus on strong universal redundancy rates for classes 
of sources over infinite alphabets. 
Note that in the latter setting, even feeble universality should not be
taken for granted: the class of {memoryless} processes on $\Np$ is not  
{feebly universal}. 

\cite{MR514346}   characterized feebly  universal
classes, and the argument was simplified by \cite{MR1281931,gyorfi1993uns}.
Recall that the entropy rate $H(P)$ of a stationary source is defined as 
$\lim_n H(P^n)/n.$ This result may be phrased in the following way.
\begin{prop} \label{prop:kieffer}
A class $\Lambda$ of stationary sources over a countable alphabet ${\mathcal{X}}$
 is feebly universal if and only if there exists a probability distribution 
 $Q\in \M_1(\mathcal{X})$ such that for every
$\PROB\in\Lambda$ with finite entropy rate, $Q$ satisfies  $\E_\PROB\log\frac{1}{Q(X_1)}<\infty$ or equivalently 
$D(\PROB^1,Q)<\infty .$
\end{prop}


Assume that $\Lambda$ is parameterized by $\Theta$ and that
$\Theta$ can be equipped with (prior) probability distributions $W$ 
in such a way that $\theta\mapsto  P^n_\theta\{A\}$ is a random variable (a measurable mapping)
  for every
$A\subseteq {\cal X}^n.$
A convenient way to derive lower bounds on $R^+(\Lambda^n)$ consists
in using the relation 
$\E_{W}[D(P_{\theta}^n, Q^n)]\leq R^+(Q^n,\Lambda^n).$

The sharpest lower bound is obtained by optimizing the prior
probability distributions,
it is called the maximin bound
$$ \sup_{W \in \mathfrak{M}_1(\Theta)} \inf_{Q^n\in
 \mathfrak{M}_1({\mathcal{X}}^n)} \E_W[D(\PROB_\theta^n, Q^n)]
.$$ 
It has been proved in a series of papers
\citep{gallager:1968,MR0465519,MR1454958}  \citep[and could also have been derived from
a general minimax theorem by ][]{MR0097026}  that such a lower
 bound is tight.  
\begin{thm} \label{thm:sion}
Let $\Lambda$ denote a class of sources over some finite or countably infinite alphabet.
 For each $n,$ the minimax redundancy   over $\Lambda$
 coincides with 
$$
R^+(\Lambda^n)=\sup_{\Theta, W \in \mathfrak{M}_1(\Theta)} \inf_{Q^n\in
 \mathfrak{M}_1({\mathcal{X}}^n)} \E_W[D(\PROB_\theta^n, Q^n)]\, ,
$$
where $\Theta$ runs over all parameterizations of countable subsets of
$\Lambda.$

If the set $\Lambda^n=\{\PROB^n~:~\PROB\in \Lambda\}$ is not pre-compact
with respect to the topology of weak convergence, then both sides  are infinite. 
A thorough account of topological issues on sets of probability measures can be found in \citep{MR1932358}.
For the purpose of this paper, it is enough to
recall that: first, a subset of a metric space is pre-compact if for any $\epsilon>0,$  it can be covered by a finite number of open balls with radius at most $\epsilon>0$; second, a sequence $(Q_n)_n$ of probability distributions  converges with respect to the topology of weak convergence toward the probability distribution $Q$ if and only if for any bounded continuous function $h$ over the support of $Q_n$'s, $\mathbbm{E}_{Q_n} h\rightarrow \mathbbm{E}_Q h.$ This topology can be metrized using the L\'evy-Prokhorov distance.  

Otherwise the  maximin and minimax average redundancies are finite and  coincide; 
moreover, the minimax redundancy is  achieved by the mixture  coding
 distribution $Q^n(.)=\int_\Theta \PROB_\theta^n(.) W(\mathrm{d}\theta)$ where
 $W$ is the least favorable prior. 
\end{thm}
Another  approach to universal coding considers   \emph{individual
sequences}  \citep[see][ and references therein]{MR1168747,cesa-bianchi:lugosi:2006}. 
Let the \emph{regret} of a coding distribution $Q^n$ on string
$\mathbf{x}\in\Np^n$ with respect to $\Lambda$ be 
$\sup_{\PROB^n\in\Lambda} \log \PROB^n (\mathbf{x})/Q^n(\mathbf{x}).$  
Taking the maximum with respect to $x\in \Np^n$, and then optimizing 
over the choice of $Q^n,$ we get the
\emph{minimax regret}: 
$$R^*(\Lambda^n) = \inf_{Q^n\in \M_1(\mathcal{X}^n)}\max_{x\in\Np^n} \sup_{P \in \Lambda}\log \frac{\PROB^n(x)}{Q^n(x)} \,  .$$ 

In order to provide proper insight, let us recall
the precise asymptotic bounds on minimax redundancy and regret  for
memoryless sources over  finite alphabets  \citep[see][ and references
  therein]{barron:clarke:1990,barron:clarke:1994,barron:rissanen:yu:1998,xie:barron:1997,xie:barron:2000,MR2097043,
catoni:2004,szpankowski:1998,MR2096987}. 

\begin{thm}\label{prop:finitealph}
Let $\mathcal{X}$ be an alphabet of $m$ symbols, and $\Lambda$ denote the class of memoryless
processes on $\mathcal{X}$ then 
\begin{eqnarray*}
\lim_n \left\{R^+(\Lambda^n)-  \frac{m-1}{2}\log\frac{n}{2\pi\mathe } \right\} &=&  \log
\left(\frac{\Gamma(1/2)^m}{\Gamma(m/2)} \right) \, \label{prop:finitealph:clarke:barron}
\\
\lim_n \left\{R^*(\Lambda^n)-  \frac{m-1}{2}\log\frac{n}{2\pi\phantom{e}}\right\} &= &\log
\left(\frac{\Gamma(1/2)^m}{\Gamma(m/2)} \right) \, .
\end{eqnarray*}
For all $n\geq 2$:
\begin{equation*}
  R^*(\Lambda^n) \leq  \frac{m-1}{2}\log{n} + 2 \, . 
\end{equation*}
\end{thm}
The last inequality is checked 
in the Appendix . 
\begin{rem}
The phenomenon pointed out in Theorem~\ref{prop:finitealph} 
 holds not only for the class of memoryless sources
over a finite alphabet but also for classes of sources that are  smoothly
parameterized by finite dimensional sets \citep[see again][]{barron:clarke:1990,barron:clarke:1994,barron:rissanen:yu:1998,xie:barron:1997,xie:barron:2000,MR2097043,catoni:2004}. 

%
\end{rem}

The minimax regret deserves further attention.
For a source class $\Lambda,$ for every $\mathbf{x}\in\mathcal{X}^n$, 
let the maximum likelihood $\hat{p}(\mathbf{x})$ be defined as~$\sup_{\PROB\in\Lambda} \PROB^n(\mathbf{x}).$
If $\sum_{\mathbf{x}\in\Np^n} \hat{p}(\mathbf{x}) <\infty,$ the \emph{Normalized
Maximum Likelihood} coding probability is well-defined and given by
$$
Q^n_{\nml}(\mathbf{x}) =  \frac{\hat{p}(\mathbf{x})}{\sum_{x\in\Np^n}  \hat{p}(\mathbf{x})}\, .
$$ 
\cite{shtarkov:1987} showed that the  \emph{Normalized Maximum Likelihood} coding probability achieves the same regret over
all strings of length $n$ and that this regret coincides with the \emph{minimax regret}: 
$$R^*(\Lambda^n)=\log \sum_{\mathbf{x}\in\Np^n} \hat{p}(\mathbf{x}).$$

%
%

Memoryless sources over finite  alphabets are special cases of
envelope classes. The latter will be of primary interest.
 
\begin{dfn}\label{dfn:envelope:class}
Let  $f$ be  a mapping from  $\Np$ to $[0,1].$ 
The envelope class $\Lambda_f$ defined by function $f$   is 
the collection of stationary  memoryless sources with first marginal
distribution dominated by $f$: 
\begin{displaymath}
\Lambda_f=\left\{ \PROB~:~~\forall x\in \mathbbm{N},\;\PROB^1\{x\}\leq
f(x)~, \text{ and } \PROB \text{ is stationary and memoryless.}
\right\}\, . 
\end{displaymath}
\end{dfn}

We will be  concerned with the following topics.
\begin{enumerate}
\item  Understanding general structural properties of minimax
  redundancy and minimax regret. 
\item Characterizing those source classes that have finite minimax
  regret. 
\item  Quantitative relations between minimax redundancy or  regret
  and integrability of the envelope function.
\item  Developing effective coding techniques for source classes with
  known non-trivial minimax redundancy rate. 
\item Developing adaptive coding schemes  for collections of 
source classes that are too large to enjoy even a weak redundancy rate.    
\end{enumerate}

The paper is organized as follows. Section~\ref{sec:props} describes some
structural properties of minimax redundancies and regrets for classes 
of stationary memoryless sources. Those properties include 
monotonicity and sub-additivity. Proposition~\ref{prop:finite} characterizes those source classes that
admit finite regret. This characterization emphasizes the role of
Shtarkov Normalized Maximum Likelihood coding probability. 
Proposition~\ref{prop:example} describes a simple source class for which
the minimax regret is infinite, while the minimax  redundancy is
finite.  Finally Proposition~\ref{prop:envfinite} asserts that such a
contrast is not possible for the so-called envelope classes. 

In Section~\ref{sec:envel:classes:generic},
Theorems~\ref{prop:upperbound} and~\ref{prop:desperate} provide quantitative relations
 between the summability properties of the envelope function
and minimax regrets and redundancies.   Those results build on the
non-asymptotic bounds on minimax redundancy derived
by \cite{xie:barron:1997}. 

Section~\ref{sec:ex} focuses on two kinds of envelope classes. This 
section serves as a benchmark for the two main results from the
preceding section. 
In Subsection~\ref{sec:power-law-envelope}, lower-bounds on minimax
redundancy and upper-bounds on minimax regret for classes defined by
envelope function  $k\mapsto 1\wedge C k^{-\alpha}$ are
described. Up to a factor $\log n$ those bounds are matching. 
In Subsection~\ref{sec:expon-envel-class}, 
lower-bounds on minimax
redundancy and upper-bounds on minimax regret for classes defined by
envelope function  $k\mapsto 1\wedge C \exp^{-\alpha k}$ are
described. Up to a multiplicative constant, those bounds coincide
and grow  like $\log^2 n.$

In Sections~\ref{sec:algo} and \ref{sec:an-adaptive-version}, we turn  to effective coding techniques  
geared toward source classes defined by power-law envelopes. In
Section~\ref{sec:algo}, we elaborate on the ideas
embodied in Proposition~\ref{prop:upperbound} from
Section~\ref{sec:props}, and combine mixture coding and Elias
penultimate code~\citep{MR0373753} to match the upper-bounds on
minimax redundancy described in Section~\ref{sec:ex}. One of the
messages from Section~\ref{sec:ex} is that the union of envelope
classes defined by power laws, does  not admit a weak redundancy rate
that grows at a rate slower than $n^{1/\beta}$ for any $\beta>1$. 
In Section~\ref{sec:an-adaptive-version}, we finally develop an adaptive
coding scheme for the union of envelope classes defined by power
laws. This adaptive coding scheme combines the censoring coding
technique developed in the preceding subsection and an estimation of
tail-heaviness. 
It shows that the union of envelope classes defined by power laws is
feebly universal. 

\section{Structural properties of the minimax redundancy and minimax regret}
\label{sec:props}
Propositions \ref{prop:folk:shtarkov},\ref{prop:increasing} and
\ref{prop:subadd}  below are sanity-check statements: they state that
when minimax redundancies and regrets are finite, as functions of
word-length, they are
non-decreasing and sub-additive. 
In order to prove them, we start by the following proposition which emphasizes the role of the ${\nml}$
coder with respect to the minimax regret. At best, it is a comment 
on Shtarkov's original work~\citep{shtarkov:1987,MR1604481}.  

\begin{prop}\label{prop:folk:shtarkov}
Let $\Lambda$ be a class of stationary memoryless sources over a countably
infinite alphabet, the minimax regret with respect to $\Lambda^n,$ $R^*(\Lambda^n)$ is finite if and
only if the normalized maximum likelihood (Shtarkov)
coding
probability $Q^n_{\nml}$ is well-defined and given by 
\begin{displaymath}
  Q^n_{\nml} (\mathbf{x}) =  \frac{\hat{p}(\mathbf{x})}{\sum_{\mathbf{y}\in
  \mathcal{X}^n} \hat{p}(\mathbf{y})}\text{ for } \mathbf{x} \in \mathcal{X}^n
\end{displaymath}
where $\hat{p}(\mathbf{x})=\sup_{\PROB \in \Lambda} \PROB^n(\mathbf{x}).$
\end{prop}

Note that the definition of $Q^n_{\nml}$ does not assume either that the
maximum likelihood is achieved on $\Lambda$ or that it is uniquely defined.

\begin{proof}
The fact that if $Q^n_{\nml}$ is well-defined, 
the minimax regret is finite and equal to $$\log\left( \sum_{\mathbf{y}\in
  \mathcal{X}^n} \hat{p}(\mathbf{y})\right) $$  is the fundamental
  observation of \cite{shtarkov:1987}. 

On the other hand, if $R^*(\Lambda^n)<\infty,$ there exists a
probability distribution $Q^n$ on $\mathcal{X}^n$ and a finite number
$r$ such that for all $\mathbf{x}\in \mathcal{X}^n,$ 
$$
\hat{p}(\mathbf{x}) \leq r \times Q^n(\mathbf{x})\, , 
$$
summing  over $\mathbf{x}$ gives 
$$
\sum_{\mathbf{x}\in
  \mathcal{X}^n} \hat{p}(\mathbf{x})\leq r < \infty \, . 
$$
\end{proof}

\begin{prop}
\label{prop:increasing}
Let $\Lambda$ denote a class of sources, then 
the minimax redundancy 
$R^+(\Lambda^n)$ and the minimax regret $R^*(\Lambda^n)$ are
  non-decreasing functions of $n$. 
\end{prop}

\begin{proof}
As far as $R^+$ is concerned, by Theorem~\ref{thm:sion},  it is enough
to check that the maximin (mutual information) lower bound  is non-decreasing. 

For any prior distribution $W$ on a parameter set $\Theta$ (recall
that 
$\{P_\theta~:~\theta\in \Theta\}\subseteq \Lambda,$ and that the mixture
coding probability $Q^n$  is defined by $Q^n(A)=\EXP_W[\PROB^n_\theta(A)]$) 
$$\EXP_W\left[ D(\PROB^{n+1}_{\theta}, Q^{n+1})\right]=
I(\theta; X_{1:n+1}) = I\left(\theta; \left(X_{1:n}, X_{n+1}\right)\right) \geq I(\theta; X_{1:n})
=\EXP_W\left[ D(\PROB^n_\theta, Q^n)\right].$$

Let us now consider the minimax regret. It is enough to consider the
 case where $R^*(\Lambda^n)$ is finite. Thus we may  rely on
 Proposition~\ref{prop:folk:shtarkov}. 
Let $n$ and $m$ be two positive integers. Let $\epsilon$ be a small 
positive real.
For any string $\mathbf{x}\in {\mathcal{X}}^{n},$
let $P_{\mathbf{x}} \in \Lambda,$ be such that $P_{\mathbf{x}}\{\mathbf{x}\}\geq 
\hat{p}(\mathbf{x})(1-\epsilon).$ Then 
\begin{eqnarray*}
  \hat{p} (\mathbf{x}x') &\geq &P_{\mathbf{x}}(\mathbf{x}) \times P_{\mathbf{x}}(x'\mid \mathbf{x}) 
\\
&\geq &\hat{p}(\mathbf{x})(1-\epsilon)  \times P_{\mathbf{x}}(x'\mid \mathbf{x}) \, .  
\end{eqnarray*}
Summing over all possible $x'\in {\cal X}$ we get 
\begin{displaymath}
\sum_{x'}  \hat{p} (\mathbf{x}x') \geq
\hat{p}(\mathbf{x})(1-\epsilon) \, . 
\end{displaymath}
Summing now over all $\mathbf{x}\in {\cal X}^n \text{ and } x'\in {\cal X},$ 
\begin{displaymath}
  \sum_{\mathbf{x}\in {\cal X}^n,x'\in {\cal X}}  \hat{p} (\mathbf{x}x') \geq
\sum_{\mathbf{x}\in {\cal X}^n}\hat{p}(\mathbf{x})(1-\epsilon) \, .
\end{displaymath}
So that  by letting $\epsilon$ tend to $0,$
\begin{displaymath}
  \sum_{\mathbf{x}\in {\cal X}^{n+1}} \hat{p}(\mathbf{x}) \geq
  \sum_{\mathbf{x}\in {\cal X}^n} \hat{p} (\mathbf{x}) \, . 
\end{displaymath}
\end{proof}
Note that the proposition holds even though $\Lambda$ is not a
collection of memoryless sources. 
This Proposition can be easily completed when dealing with
memoryless sources. 
\begin{prop}
\label{prop:subadd}
If $\Lambda$ is a class of stationary memoryless sources,  then 
the functions $n\mapsto R^+(\Lambda^n)$ and $n\mapsto R^*(\Lambda^n)$
are either infinite or
sub-additive.
\end{prop}

\begin{proof} Assume that $R^+(\Lambda^n)<\infty$. 
Here again, given Theorem~\ref{thm:sion}, in order to
  establish sub-additivity for $R^+,$ it is enough
  to check the property for the maximin lower bound.  Let $n, m$ be
  two positive integers, and $W$ be any prior on $\Theta$ 
  (with $\{P_\theta~:~\theta\in \Theta\}\subseteq\Lambda$). 
As sources from $\Lambda$ are memoryless, $X_{1:n}$ and $X_{n+1:n+m}$ are independent conditionally on $\theta$ and thus 
\begin{eqnarray*}
\lefteqn{I\left(X_{n+1:n+m}; \theta |X_{1:n}\right)}\\
 & = & H\left(X_{n+1:n+m}|X_{1:n}\right)- H\left(X_{n+1:n+m}|X_{1:n}, \theta\right)\\
& = & H\left(X_{n+1:n+m}|X_{1:n}\right)- H\left(X_{n+1:n+m}| \theta\right)\\
& \leq & H\left(X_{n+1:n+m}\right)- H\left(X_{n+1:n+m}| \theta\right)\\
& = & I\left(X_{n+1:n+m}; \theta \right)\, .
\end{eqnarray*}
Hence, using the fact that under each $\PROB_\theta,$ the process  $(X_n)_{n\in\Np}$ is stationary:
\begin{eqnarray*}
I\left(X_{1:n+m}; \theta\right) & = & I\left(X_{1:n}; \theta\right) +
I\left(X_{n+1:n+m}; \theta |X_{1:n}\right)\\
& \leq & I\left(X_{1:n}; \theta\right) + I\left(X_{n+1:n+m};\theta
\right)\\
& = & I\left(X_{1:n}; \theta\right) + I\left(X_{1:m}; \theta
\right).
\end{eqnarray*}
Let us now check the sub-additivity of the minimax regret. Suppose that $R^*(\Lambda^1)$ is finite.
For any $\epsilon>0,$
for $\mathbf{x}\in {\cal X}^{n+m},$
let $\PROB \in \Lambda$ be such that 
$(1-\epsilon) \hat{p}(\mathbf{x})\leq \PROB^{n+m}(\mathbf{x}).$
As for $\mathbf{x}\in {\cal X}^n$ and $\mathbf{x}'\in {\cal X}^m,$ $\PROB^{n+m}(\mathbf{x}\mathbf{x}')=
\PROB^n(\mathbf{x})\times \PROB^m(\mathbf{x}'), $
 we have for any $\epsilon>0,$ and any $\mathbf{x}\in {\cal
   X}^n,\mathbf{x}'\in {\cal X}^m$
\begin{displaymath}
  (1-\epsilon) \hat{p}(\mathbf{x}\mathbf{x}') \leq \hat{p}(\mathbf{x})\times
  \hat{p}(\mathbf{x}')\, .
\end{displaymath}
Hence, letting  $\epsilon$ tend to $0,$ and summing over all
$\mathbf{x}\in {\cal X}^{n+m}$: 
\begin{eqnarray*} 
\lefteqn{R^* \left(\Lambda^{n+m}\right)}\\
 & = & \log \sum_{\mathbf{x}\in {\mathcal{X}}^{n}\mathbf{x}'\in {\mathcal{X}}^{m}}
\hat{p}\left(\mathbf{x}\mathbf{x}'\right) \\
 & \leq &  \log \sum_{\mathbf{x}\in {\mathcal{X}}^{n}}
\hat{p}(\mathbf{x}) + \log \sum_{\mathbf{x}\in {\mathcal{X}}^{m}}
\hat{p}\left(\mathbf{x}'\right)\\
& = & R^*\left(\Lambda^{n}\right) +R^*\left(\Lambda^{m}\right).
\end{eqnarray*}
\end{proof}
\begin{rem}
Counter-examples  witness the fact that subadditivity of redundancies does not hold 
in full generality. 
\end{rem}
The Fekete Lemma
\citep[see][]{dembo:zeitouni:1998} leads to:
\begin{cor} \label{cor:fekete}
Let $\Lambda$ denote a class of stationary memoryless sources over a countable alphabet.
 For both minimax redundancy $R^+$ and minimax regret $R^*$, 
$$\lim_{n\to\infty}\frac{R^+\left(\Lambda^n\right)}{n} =
\inf_{n\in\Np}\frac{R^+\left(\Lambda^n\right)}{n} \leq
R^+\left(\Lambda^1\right)\, , $$
and 
$$\lim_{n\to\infty}\frac{R^* \left(\Lambda^n\right)}{n} =
\inf_{n\in\Np}\frac{R^* \left(\Lambda^n\right)}{n} \leq
R^* \left(\Lambda^1\right)\, . $$
\end{cor}

\mbox{}\vspace{1.cm}\\

Hence, in order to prove that $R^+\left(\Lambda^n\right)<\infty$
(respectively $R^*\left(\Lambda^n\right)<\infty$),
it is enough to check that $R^+\left(\Lambda^1\right)<\infty$
(respectively $R^*\left(\Lambda^1\right)<\infty$).

The following Proposition combines
Propositions~\ref{prop:folk:shtarkov},~\ref{prop:increasing}
and~\ref{prop:subadd}.
It can be rephrased as follows: a class of memoryless sources admits
a non-trivial strong minimax regret if and only if Shtarkov ${\nml}$  coding 
probability is well-defined for $n=1.$

\begin{prop}
\label{prop:finite}
Let $\Lambda$ be a class of stationary memoryless sources over a countably
infinite alphabet. Let $\hat{p}$ be defined by
$\hat{p}(x)=\sup_{\PROB\in \Lambda}\PROB\{x\}.$ The minimax regret with respect to $\Lambda^n$ is finite if and
only if the normalized maximum likelihood (Shtarkov)
coding
probability is well-defined and :  
 $$R^*(\Lambda^n)<\infty 
\equivaut \sum_{x\in\Np} \hat{p}(x) <\infty.$$ 
\end{prop}
\mbox{}\vspace{1cm}\\

\begin{proof}
The direct part follows from Proposition~\ref{prop:folk:shtarkov}.

For the converse  part,
if $\sum_{x\in\Np} \hat{p} (x) = \infty$, then 
$R^*(\Lambda^1)=\infty$ and from Proposition \ref{prop:increasing},
$R^*(\Lambda^n)=\infty$ for every positive integer  $n$.
\end{proof}

When dealing with smoothly parameterized classes of  sources over
finite alphabets
\citep[see][]{barron:rissanen:yu:1998,xie:barron:2000} or even with 
the massive classes defined by renewal sources 
\citep{csiszar:shields:1996}, the minimax  regret and minimax
redundancy are usually of the same order of magnitude (see
Theorem~\ref{prop:finitealph} and comments in the Introduction).
This can not be taken for granted when 
dealing with classes of stationary memoryless sources over a
countable alphabet. 
\begin{prop}
\label{prop:example}
Let $f$ be a positive, strictly decreasing function defined on $\mathbbm{N}$ such that $f(1)<1$.
For $k\in\mathbbm{N}$, let $p_k$ be the probability mass function on $\mathbbm{N}$ defined by:
$$p_k(l) = \left\{\begin{array}{ccl}
{1-f(k)} & \text{if} & l=0 ; \\
f(k) & \text{if} & l=k;\\
0 &  & \text{otherwise.}
\end{array}\right.$$
Let $\Lambda^1  = \{p_1, p_2, \ldots\},$ let $\Lambda$ be the class of
stationary memoryless sources with first marginal $\Lambda^1.$ The
finiteness of the 
minimax redundancy with respect to $\Lambda^n$ depends on the
limiting behavior of $f(k)\log k$:
for every positive integer $n$: 
 $$  f(k)\log k \to_{k\to\infty} \infty \Leftrightarrow
 R^+\left(\Lambda^n\right)=\infty\, .$$
\end{prop}
\vspace{.5cm}
\begin{rem}
When $f(k)=\frac{1}{\log k}$, the minimax  redundancy
$R^+(\Lambda^n)$ is finite for all $n.$ 
Note, however that this does  not warrant the existence of a
non-trivial strong universal redundancy rate. 
However, as $\sum_k f(k) = \infty$, minimax  regret is
infinite by Proposition~\ref{prop:finite}. 

A similar result appears in the discussion of Theorem 3 in
\citep{MR1604481} where classes with finite minimax redundancy and
infinite minimax regret are called irregular. 

We will be able to refine those observations after the statement of
Corollary~\ref{cor:subadd:regret}.  

\end{rem}
\begin{proof}
Let us first prove the direct part. Assume  that $f(k)\log k
\to_{k\to\infty} \infty$. 
In order to check that $R^+(\Lambda^1)=\infty,$  we resort to the
mutual information lower bound (Theorem~\ref{thm:sion}) 
and describe an appropriate  collection of
Bayesian games. 

Let $m$ be a positive integer and let $\theta$ be   uniformly
distributed
over $\{1,2,\ldots,m\}$.  Let $X$  be distributed according to
$p_k$ conditionally on $\theta=k.$
Let $Z$ be the random  variable equal to $1$ if $X=\theta$ and equal to $0$
otherwise.
Obviously, $H(\theta|X, Z=1)=0$; moreover, as $f$ is assumed to be
non-increasing, $\PROB(Z=0|\theta=k) = 1-f(k)\leq 1-f(m)$ and thus: 
\begin{eqnarray*}
H(\theta|X) & = & H(Z|X) + H(\theta|Z,X)\\
 & \leq & 1 + \PROB(Z=0)H(\theta|X, Z=0)\\
&& + \PROB(Z=1)H(\theta|X, Z=1)\\
 & \leq & 1 +  \left(1-f(m)\right) \log m.
\end{eqnarray*}
Hence, 
\begin{eqnarray*}
R^+(\Lambda^1) & \geq & I(\theta, X)\\
 & \geq & \log m - \left(1-f(m)\right) \log m\\
 & = & f(m) \log m
\end{eqnarray*}
which grows to infinity with $m$, so that as announced $R^+(\Lambda^1)=\infty$.

Let us now prove the converse part. 
Assume that the sequence $(f(k)\log k)_{k\in \mathbbm{N}_+} $ is
upper-bounded by some constant $C.$ In order to check that
$R^+(\Lambda^n)<\infty,$ for all $n,$ by
Proposition~\ref{prop:subadd}, it is enough to check that
$R^+(\Lambda^1_f)<\infty,$  
and thus, it is enough to exhibit a
probability distribution $Q$ over $\mathcal{X}=\mathbbm{N}$  such that 
$\sup_{\PROB\in \Lambda^1} D(\PROB,Q)<\infty. $

Let $Q$ be defined by $Q(k)= A/((1\vee (k (\log k)^2))$  for $k\geq 2$,
$Q(0),Q(1) >0$ where $A$ is a normalizing constant that ensures that 
$Q$ is a probability distribution over $\mathcal{X}.$

Then for any $k\geq 3$ (which warrants $k (\log k)^2>1$), letting $P_k$
be the probability defined by the probability mass function $p_k$: 
\begin{eqnarray*}
\lefteqn{  D(P_k , Q)} \\
&=& \left(1-f(k)\right) \log \frac{(1-f(k))}{Q(0)}
+ f(k) \log \left( \frac{f(k) k (\log k)^2}{A} \right)
\\
& \leq & -\log Q(0) + C + f(k)  \left( 2  \log^{(2)}(k) -
\log(A)\right)
\\
& \leq & C +\log \frac{C^2}{A\, Q(0)} \, .
\end{eqnarray*}
This is enough to conclude that $$R^+(\Lambda^1)\leq  \left( C+\log
\frac{C^2}{A\, Q(0)}\right)\vee D(P_1,Q) \vee D(P_2,Q) <\infty \, .$$ 
\end{proof}
\begin{rem}
Note that the coding probability used in the proof of the converse
part of the proposition corresponds to one of the simplest prefix codes
for integers proposed  by \cite{MR0373753}. 
\end{rem}

The following theorem shows that,  as far as  envelope classes are
concerned (see Definition \ref{dfn:envelope:class}), minimax redundancy and minimax regret are either both
finite of both infinite. This is indeed much less precise than the 
relation stated in Theorem~\ref{prop:finitealph} about classes of sources on finite alphabets.

\begin{thm}
\label{prop:envfinite}
Let $f$ be a non-negative function from  $\mathbbm{N}_+$ to $[0,1]$,
let $\Lambda_f$ be the class of stationary memoryless sources defined by envelope $f.$
Then 
 $$
 R^+\left(\Lambda_f^{n}\right)<\infty \equivaut
 R^*\left(\Lambda_f^{n}\right)<\infty \, . $$ 
\end{thm}
\mbox{}\vspace{.5cm}\\
\begin{rem}
  We will refine this result after the statement of
  Corollary~~\ref{cor:subadd:regret}. 
\end{rem}

Recall from Proposition~\ref{prop:finite}  that $
 R^*\left(\Lambda_f^{n}\right)<\infty \equivaut
\sum_{k\in\Np} f(k) < \infty.$ 

\begin{proof}

In order to check 
that 
$$
 \sum_{k\in\Np} f(k) = \infty \Rightarrow
 R^+\left(\Lambda_f^{n}\right) =
 \infty \, ,
$$ it is enough to check that if $ \sum_{k\in\Np} f(k) = \infty ,$ the
envelope class contains an
infinite collection of mutually singular sources. 

Let the infinite sequence  of integers $\left( h_i\right)_{i\in
  \mathbbm{N}}$ be defined recursively by $h_0=0$ and  
$$
h_{i+1}=\min \left\{ h ~:~ \sum_{k=h_i+1}^h f(k) >1 \right\} \, . 
$$
The memoryless source $P_i$ is defined by its first marginal $P_i^1$
which is given by 
$$
P_i^1(m) = \frac{f(m)}{\sum_{k=h_i+1}^{h_{i+1}} f(k)} \text{ for }
m\in\{p_i+1, ...,p_{i+1}\} \, . 
$$ Taking any prior
with infinite Shannon entropy over the $\{P_i^1~;~i\in \Np\} $ shows
that 
$$
R^+\left(\{P_i^1~;~i\in \Np\}\right) =\infty \, . 
$$
\end{proof}


\section{Envelope classes}
\label{sec:envel:classes:generic}

The next two theorems establish quantitative relations between
minimax redundancy and regrets and the shape of the envelope
function. Even though the two theorems deal with general envelope
functions, the reader might appreciate to have two concrete examples 
of envelope in mind:  exponentially decreasing envelopes of the form 
$C e^{-\alpha k}$ for appropriate $\alpha$ and $C$, and power-laws of
the form $C k ^{-\alpha}$ again for appropriate $\alpha$  and $C.$
The former family of envelope classes extends the class of 
sources over finite (but unknown) alphabets. 
The first theorem holds for any class of memoryless sources.

\begin{thm}
\label{prop:upperbound}
If $\Lambda$ is a class of memoryless sources, 
 let the tail function
$\bar{F}_{\Lambda^1}$ be defined by 
$\bar{F}_{\Lambda^1}(u)=\sum_{k>u} \hat{p}(k),$
then:
\begin{displaymath}
R^*(\Lambda^n) \leq \inf_{u : u\leq n} \, \left[ n
  \bar{F}_{\Lambda^1}(u)\log e + \frac{u-1}{2}\log n 
  \right] + 2\, .
\end{displaymath}
\end{thm}
\mbox{}\vspace{1cm}\\

Choosing a sequence $(u_n)_n$ of positive integers in such a way that $u_n\rightarrow
\infty$ while $(u_n\log n) /n\rightarrow 0,$ this theorem allows to complete Proposition~\ref{prop:finite}. 
\begin{cor}\label{cor:subadd:regret}
Let $\Lambda$ denote a class of memoryless sources, 
then the following holds:
$$
R^*(\Lambda^n)<\infty \equivaut   R^*(\Lambda^n) = o(n) \text{ and } R^+(\Lambda^n)=o(n) \, .
$$
\end{cor}
\mbox{}\vspace{.5cm}\\

\begin{rem}
We may now have a second look at Proposition~\ref{prop:example} and
Theorem~\ref{prop:envfinite}. In the setting of
Proposition~\ref{prop:example}, this Corollary asserts that if $\sum_k
f(k) <\infty ,$ for the source
class defined by $f$, a non-trivial strong redundancy rate  exists. 

On the other hand, this corollary complements Theorem~\ref{prop:envfinite} by asserting
that  envelope classes have either non-trivial strong redundancy rates
or infinite minimax redundancies. 
\end{rem}

\begin{rem}
Again,   this statement has to be connected with related propositions from 
\cite{MR1604481}. The last paper  establishes bounds
on  minimax redundancy using geometric properties of the source
class under Hellinger metric. For example, 
Theorem 4 in \citep{MR1604481} relates minimax redundancy and the
metric dimension of the set $\Lambda^n$ with respect to the Hellinger
metric (which coincides with $L_2$ metric between the square roots of
densities) under the implicit assumption that sources lying in small
Hellinger balls have finite relative entropy (so that upper bounds in
Lemma 7 there are finite). Envelope classes may not satisfy this
assumption. Hence, there is no easy way to connect
Theorem~\ref{prop:upperbound}
and results from~\citep{MR1604481}. 
\end{rem}
\begin{proof}(Theorem~\ref{prop:upperbound}.)
Any integer $u$ defines a decomposition of a string $\mathbf{x}\in \Np^n$  into two non-contiguous
substrings: a substring $\mathbf{z}$ made of the $m$ symbols from
$\mathbf{x}$ that are larger than $u$, and one substring
$\mathbf{y}$ made of the $n-m$ symbols that are smaller than $u.$
\begin{eqnarray*}
\lefteqn{\sum_{\mathbf{x}\in\Np^n} \hat{p}(\mathbf{x})}\\
 & \stackrel{(a)}{=} & 
\sum_{m=0}^{n} \binom{n}{m} \sum_{\mathbf{z}\in \{u+1,...\}^m} \sum_{\mathbf{y}\in
\{1,2,\ldots,u\}^{n-m}}  \hat{p}\left(\mathbf{z}\mathbf{y}\right) \label{prup:l1}\\
& \stackrel{(b)}{\leq} & \sum_{m=0}^{n} \binom{n}{m}
\sum_{\mathbf{z}\in \{u+1,...\}^m}
\prod_{i=1}^m \hat{p}\left(\mathbf{z}_i\right) 
\sum_{\mathbf{y}\in\{1,2,\ldots,u\}^{n-m}}
\hat{p}\left(\mathbf{y}\right)
\\
& \stackrel{(c)}{\leq} & \left(\sum_{m=0}^{n} \binom{n}{m} \bar{F}_{\Lambda^1}(u)^m \right)
\left(\sum_{\mathbf{y}\in\{1,2,\ldots,u\}^n}
\hat{p}\left(\mathbf{y}\right)\right)
\\
& \stackrel{(d)}{\leq} &  \left(1+\bar{F}_{\Lambda^1}(u)\right)^n
2^{\frac{u-1}{2}\log{n}{}+ 2} \, . 
\end{eqnarray*}
Equation  (a) is obtained by  reordering the symbols in the strings,
Inequalities (b) and (c)  follow respectively from Proposition
\ref{prop:subadd} and Proposition \ref{prop:increasing}.
Inequality (d)  is a direct consequence  of the last inequality in Theorem~\ref{prop:finitealph}.

Hence, 
\begin{eqnarray*}
R^*(\Lambda^n) & \leq & n\log \left(1+\bar{F}_{\Lambda^1}(u)\right) +
\frac{u-1}{2}\log {n} +2\\
& \leq & n\bar{F}_{\Lambda^1}(u)\log e + \frac{u-1}{2}\log {n} +2
\end{eqnarray*}

\end{proof}

The next theorem complements the upper-bound on minimax regret for
envelope classes (Theorem~\ref{prop:upperbound}). It describes a general
lower bound on minimax redundancy for envelope classes. 
 
\begin{thm}
  \label{prop:desperate}
Let $f$ denote a non-increasing, summable envelope function. 
 For any integer $p,$ let $c(p)=\sum_{k=1}^p f(2k).$
Let $c(\infty)=\sum_{k\geq 1} f(2k).$ Assume furthermore that
$c(\infty)>1.$ Let $p\in\Np$ be such that $c(p)>1.$ 
Let  $n\in\Np$, $\epsilon>0$ and $\lambda\in]0,1[$ be such
    that 
\begin{math}
  n>\frac{c(p)}{f(2p)}\frac{10}{\epsilon(1-\lambda)}\,. 
\end{math}
Then
\begin{equation*}
  R^+(\Lambda^n_f) \geq C(p,n,\lambda, \epsilon) \sum_{i=1}^p\left(
  \frac{1}{2} \log \frac{n\left(1-\lambda\right)\pi f(2i)}{2c(p)e}-\epsilon\right),
\end{equation*} 
where 
\begin{math}
C(p,n,\lambda, \epsilon) = 
\frac{1}{1+\frac{c(p)}{\lambda^2nf(2p)}}\left(1-
    \frac{4}{\pi}\sqrt{\frac{5 c(p)}{\left(1-\lambda\right)\epsilon n
    f(2p)}}\right).
\end{math}
\end{thm}

Before proceeding to the proof, let us mention the following
non-asymptotic bound from \cite{xie:barron:1997}. 
Let $m^{\ast}_n$ denote the Krichevsky-Trofimov distribution over $\{ 0, 1
\}^n .$ That is, for any $\mathbf{x}\in \{0,1\}^n,$ such that
$n_1=\sum_{i=1}^n \mathbf{x}_i$ and $n_0=n-n_1$
\[ m_n^{\ast} \left( \mathbf{x} \right) = \frac{1}{\pi} \int_{[ 0, 1 ]}
\theta^{n_1 -
   1 / 2} ( 1 - \theta )^{n_0 - 1 / 2} \mathd \theta . \]
\begin{lem}\citep[][Lemma 1]{xie:barron:1997}\label{lem:barron:xie}
  For any $\varepsilon > 0$, there exists a $c ( \varepsilon )$ such that for
  $n > 2 c ( \varepsilon )$ the following holds uniformly over $\theta \in [ c
  ( \varepsilon ) / n, 1 - c ( \varepsilon ) / n ]$:
  \[ \left| D \left( p^n_{\theta}, m_n^{\ast} \right) - \frac{1}{2} \log
     \frac{n}{2 \pi \mathe} - \log \pi \right| \leqslant \varepsilon . \]
  The bound $c ( \varepsilon )$ can be chosen as small as $5 / \varepsilon .$
\end{lem}

\begin{proof}
The proof is organized in the following way. A prior probability 
distribution is first designed in such a way that it is supported
by probability distributions that satisfy the envelope condition, have support equal to $\{1,,\ldots,2p\},$ and most importantly enjoys the following property. 
Letting $\mathbf{N}$ be the random vector from $\mathbbm{N}^p$ defined by 
$\mathbf{N}_i(\mathbf{x})= |\{j~:~\mathbf{x}_j\in \{2i-1,2i\}\}|,$ where $\mathbf{x}\in \mathcal{X}^n$, letting $Q^*$ be the mixture distribution over $\mathcal{X}^n$ defined by the prior, then for any $P$ in the support of the prior, $P^n/Q^*= P^n\{.\mid \mathbf{N}\}/Q^*\{.\mid \mathbf{N}\}.$ This property will provide a handy way to bound the capacity of the channel. 

Let $f,n,p,\epsilon\text{ and } \lambda$    be as in the statement of the
theorem. 
Let us first define a prior probability  on $\Lambda^1_f.$  For
each integer $i$ between $1$ and $p,$ let $\mu_i$ be defined as
\begin{displaymath}
   \mu_i= \frac{f(2i)}{c(p)} \, . 
\end{displaymath}
This ensures that the sequence $(\mu_i)_{i\leq p}$ defines a probability mass function over $\{1,\ldots,p\}.$ 
Let $\boldsymbol{\theta}=(\theta_i)_{1\leq i\leq p}$  be a collection  of independent random
variables each distributed according to a Beta distribution with
parameters $(1/2,1/2)$.  The prior probability $W=\otimes_{i=1}^{p}W_{i}$ for $\boldsymbol{\theta}$ on
$[0,1]^{p}$ has thus density
$w$ given by
\begin{displaymath}
  w(\boldsymbol{\theta}) = \frac{1}{\pi^p} \prod_{i=1}^p
  \left({\theta}_i^{-1/2}(1-{\theta}_i)^{-1/2}
  \right) \, . 
\end{displaymath}
The memoryless source parameterized by
$\boldsymbol{\theta}$
is defined by the probability mass function $p_{\boldsymbol{\theta}}(2i-1)= \theta_i\,
\mu_i$ and $p_{\boldsymbol{\theta}}(2i)=(1-\theta_i)\mu_i$ for $i:~1\leq i\leq p$ and $p_{\boldsymbol{\theta}}(j)=0$ for $j>2i.$
Thanks to the condition $c(p)>1,$ this probability mass function
satisfies the envelope condition.

For $i\leq p,$ let the random variable $N_i$ (resp. $N^0_i$) be defined as the number
of occurrences of $\{2i-1,2i\}$ (resp. $2i-1$) in the sequence $\mathbf{x}.$   Let
$\mathbf{N}$ (resp. $\mathbf{N}^0$) denote the random vector
$N_1,\ldots,N_p$ (resp. $N^0_1,\ldots,N^0_p$).
If a sequence $\mathbf{x}$ from $\{1,\ldots,2p\}^n$ contains $n_i =N_i(\mathbf{x})$ symbols from
$\{2i-1,2i\}$ for each $i\in \{1,\ldots,p\},$ and if for each such
$i,$
the sequence contains $n_{i}^0=N^0_i(\mathbf{x})$ ($n_i^1$) symbols equal to $2i-1$
(resp. $2i$) then
\begin{displaymath}
  P^n_{\boldsymbol{\theta}}(\mathbf{x}) = \prod_{i=1}^p \left( \mu_i^{n_i}\, 
  \theta_i^{n_i^0} \, (1-\theta_i)^{n_i^1} \right) \, .
\end{displaymath} 
Note that when the source $\boldsymbol{\theta}$
is picked according to the prior  $W$ and the sequence
 $X_{1:n}$ picked according 
to $P_{\boldsymbol{\theta}}^n,$
the   random vector $\mathbf{N}$  is multinomially distributed with
parameters $n$ and $(\mu_1,\mu_2,\ldots,\mu_p),$ so the distribution
of $\mathbf{N}$ does not depend on the outcome of $\boldsymbol{\theta}.$
Moreover, conditionally on $\mathbf{N},$ the conditional probability
$P_{\boldsymbol{\theta}}^n\left\{\cdot\mid \mathbf{N}\right\}$  
is a  product distribution: 
\begin{displaymath}
  P^n_{\boldsymbol{\theta}}(\mathbf{x}\mid \mathbf{N}) = \prod_{i=1}^p \left(\theta_i^{n_i^0} \, (1-\theta_i)^{n_i^1} \right) \, .
\end{displaymath} 
In statistical parlance, the random vectors $\mathbf{N}$ and
$\mathbf{N}^0$ form a sufficient statistic for $\boldsymbol{\theta}.$

Let $Q^*$ denote the mixture distribution on $\mathbbm{N}_+^n$ induced by $W$:
\begin{displaymath}
  Q^*(\mathbf{x})
  =\mathbbm{E}_W\left[P^n_{\boldsymbol{\theta}}(\mathbf{x})\right] \, ,
\end{displaymath}
and, for each $n,$ let $m^*_n$ denote the  Krichevsky-Trofimov mixture over
$\{0,1\}^n,$
then 
\begin{displaymath}
  Q^*(\mathbf{x}) = \prod_{i=1}^p \left( \mu_i^{n_i} \,
    m^*_{n_i}(0^{n_i^0}1^{n_i^1})\right) 
  \, . 
\end{displaymath}
For a given value of $\mathbf{N},$ the conditional probability
$Q^*\left\{\cdot \mid \mathbf{N}\right\}$ is also a product
distribution:
\begin{displaymath}
  Q^*\left(\mathbf{x}\mid \mathbf{N}\right) = \prod_{i=1}^p\,
  m^*_{N_i}(0^{n^0_i}1^{n^1_i}) \, , 
\end{displaymath}
so, we will be able to rely on: 
\begin{displaymath}
  \frac{P^n_{\boldsymbol{\theta}}(\mathbf{x})}{Q^*(\mathbf{x})}
= \frac{P^n_{\boldsymbol{\theta}}(\mathbf{x}\mid \mathbf{N})}{Q^*(\mathbf{x}\mid
  \mathbf{N})} \, . 
\end{displaymath}

Now, the average redundancy of $Q^*$ with respect to
$P^n_{\boldsymbol{\theta}}$   can be rewritten in a handy way. 
\begin{eqnarray*}
 {\mathbbm{E}_W\left[ D\left(P^n_{\boldsymbol{\theta}}, Q^*
    \right) \right]} 
& =& \mathbbm{E}_W\left[ \mathbbm{E}_{P^n_{\boldsymbol{\theta}}}\left[
\log 
\frac{P^n_{\boldsymbol{\theta}}(X_{1:n}\mid \mathbf{N})}{Q^*(X_{1:n}\mid
  \mathbf{N})}
\right]
 \right]
\\
& & \text{from the last equation,}
\\
& =& \mathbbm{E}_W\left[ \mathbbm{E}_{P^n_{\boldsymbol{\theta}}}\left[
\mathbbm{E}_{P^n_{\boldsymbol{\theta}}}\left[
\log 
\frac{P^n_{\boldsymbol{\theta}}(X_{1:n}\mid \mathbf{N})}{Q^*(X_{1:n}\mid
  \mathbf{N})} \big\vert \mathbf{N}
\right]
\right]
 \right]
\\
& =& 
\mathbbm{E}_{W} 
\left[ 
\mathbbm{E}_{P^n_{\boldsymbol{\theta}}}\left[
D\left(P^{n}_{\boldsymbol{\theta}}(\cdot\mid \mathbf{N}), Q^*_{n}(\cdot \mid \mathbf{N})\right)
    \right] \right]
\\
& =& 
\mathbbm{E}_{W} 
\left[ 
\mathbbm{E}_{\mathbf{N}}\left[
D\left(P^{n}_{\boldsymbol{\theta}}(\cdot\mid \mathbf{N}), Q^*(\cdot \mid \mathbf{N}) \right)
    \right] \right]
\\
& & \text{as the distribution of $\mathbf{N}$ does not depend on
  $\boldsymbol{\theta}$,}
\\
&=& 
\mathbbm{E}_{\mathbf{N}} 
\left[ 
\mathbbm{E}_{W}\left[
D\left(P^{n}_{\boldsymbol{\theta}}(\cdot\mid \mathbf{N}), Q^*(\cdot\mid \mathbf{N}) \right)
    \right] \right]
\, 
\\
&& \text{by Fubini's Theorem.}
\end{eqnarray*}
We may  develop $D\left(P^{n}_{\boldsymbol{\theta}}(\cdot\mid
  \mathbf{N}), Q^*(\cdot\mid \mathbf{N}) \right)$
for a given value of $\mathbf{N}=(n_1, n_2, \ldots,n_p).$ As both
  $P^n_{\boldsymbol{\theta}}(\cdot\mid \mathbf{N})$
and $Q^*(\cdot\mid \mathbf{N})$ are product distributions on 
$\prod_{i=1}^p \left( \{2i-1,2i\}^{n_i}\right), $ we have 
\begin{eqnarray*}
\mathbbm{E}_W\left[D\left(P^{n}_{\boldsymbol{\theta}}(\cdot\mid \mathbf{N}),
  Q^*(\cdot\mid \mathbf{N}) \right)  \right]
& =&  
  \mathbbm{E}_W\left[\sum_{i=1}^p D\left(P^{n_i}_{{\theta}_i},
  m^*_{n_i} \right)  \right] 
\\
& = & 
\sum_{i=1}^p \, \mathbbm{E}_{W_i}\left[ D\left(P^{n_i}_{{\theta}_i},
  m^*_{n_i} \right)  \right] 
\, . 
\end{eqnarray*}
The minimal average redundancy of $\Lambda^n_f$ with respect to the
mixing distribution $W$ is thus finally given by: 
\begin{eqnarray}
{\mathbbm{E}_W\left[ D\left(P^n_{\boldsymbol{\theta}}, Q^*
    \right) \right]} \notag
& =& 
\mathbbm{E}_{\mathbf{N}}
\left[ \sum_{i=1}^p \, \mathbbm{E}_{W_i}
\left[ D\left(P^{n_i}_{{\theta}_i},m^*_{n_i} \right)  \right] \right]
\\ \notag
 & =& \sum_{i=1}^p 
\mathbbm{E}_{\mathbf{N}}
\left[  \mathbbm{E}_{W_i}
\left[ D\left(P^{n_i}_{{\theta}_i},  m^*_{n_i} \right)  \right] \right]
\\ 
& = & \sum_{i=1}^p 
\sum_{n_i=0}^n \binom{n}{n_i} \mu_i^{n_i} (1-\mu_i)^{n-n_i}
\mathbbm{E}_{W_i}\left[D\left(P_{\theta_i}^{n_i}, m_{n_i}^*\right)\right]  \label{eq:passansmal}\, . 
\end{eqnarray}
Hence, the minimal redundancy of $\Lambda^n_f$  with respect to
 prior probability $W$ is a weighted average of redundancies of
 Krichevsky-Trofimov mixtures over binary strings with different
 lengths. 

 At some place, we will use 
 the Chebychef-Cantelli inequality
\citep[see][]{devroye:gyorfi:lugosi:1996} which asserts that  for a square-integrable random variable:
\begin{displaymath}
  \Pr \left\{ X \leq \mathbbm{E}[X]-t\right\}\leq
  \frac{\text{Var}(X)}{\text{Var}(X)+t^2} \, . 
\end{displaymath}

Besides, note that for all $\epsilon<\frac{1}{2}$, 
\begin{equation}\label{eq:obs:triv:beta}
\int_\epsilon^{1-\epsilon} \frac{\mathd x}{\pi\sqrt{x(1-x)}}
>1-\frac{4}{\pi}\sqrt{\epsilon}.
\end{equation}
Now, the proposition is derived by  processing the right-hand-side of
Equation (\ref{eq:passansmal}). \\
Under condition $\left(1-\lambda\right)n\frac{f(2p)}{c(p)}>\frac{10}{\epsilon}$, we
have  $\frac{5}{n_i\epsilon}<\frac{1}{2}$ for all $i\leq p$ such that $ n_i\geq \left(1-\lambda\right)n\mu_i$.
Hence,
\begin{eqnarray*}
\lefteqn{\mathbbm{E}_W\left[ D\left(P^n_{\boldsymbol{\theta}}, Q^*
    \right) \right]}
\\
& \geq & \sum_{i=1}^p 
\sum_{n_i\geq (1-\lambda)n\mu_i}^n \binom{n}{n_i} \mu_i^{n_i}
(1-\mu_i)^{n-n_i}
\int_{0}^1 \frac{  
D\left(P_{\theta_i}^{n_i}, m_{n_i}^*\right)}{
\sqrt{\theta_i(1-\theta_i)}}\, \mathd \theta_i
\\
& \geq & \sum_{i=1}^p 
\sum_{n_i\geq (1-\lambda)n\mu_i}^n \binom{n}{n_i} \mu_i^{n_i}
(1-\mu_i)^{n-n_i}
\int_{\frac{5}{n_i \, \epsilon}}^{1-\frac{5}{n_i \,\epsilon}}
\frac{\frac{1}{2} \log \frac{n_i}{2\pi\mathe} +\log \pi-\epsilon }{ 
\pi\sqrt{\theta_i(1-\theta_i)}}\,\mathd \theta_i 
\\
&&\text{ by Proposition~\ref{lem:barron:xie} from \cite{xie:barron:1997}}
\\
& \geq & 
\sum_{i=1}^p 
\sum_{n_i\geq (1-\lambda)n\mu_i}^n \binom{n}{n_i} \mu_i^{n_i}
(1-\mu_i)^{n-n_i}
\left(1- \frac{4}{\pi}\sqrt{\frac{5}{n_i\epsilon}}\right) 
\left(\frac{1}{2} \log \frac{n_i}{2\pi\mathe} +\log
  \pi-\epsilon  \right) 
\\
& & \text{from~(\ref{eq:obs:triv:beta})}
\\
& \geq & 
\sum_{i=1}^p 
\sum_{n_i\geq (1-\lambda)n\mu_i}^n \binom{n}{n_i} \mu_i^{n_i}
(1-\mu_i)^{n-n_i}
\left(1- \frac{4}{\pi}\sqrt{\frac{5}{\left(1-\lambda\right)n\mu_i\epsilon}}\right) 
\left(\frac{1}{2} \log \frac{n (1-\lambda)\mu_i}{2\pi\mathe} +\log
  \pi-\epsilon  \right) \\
& & \text{using monotonicity of $x\log x$}
\\
& \geq & 
\sum_{i=1}^p 
\frac{1}{1+\frac{1-\mu_i}{n\,\mu_i\,\lambda^2}}
\left(1- \frac{4}{\pi}\sqrt{\frac{5}{\left(1-\lambda\right)n\mu_i\epsilon}}\right) 
\left(\frac{1}{2} \log \frac{n (1-\lambda)\mu_i}{2\pi\mathe} +\log
  \pi-\epsilon  \right) 
\\
&& \text{invoking the Chebychef-Cantelli inequality,}
\\
& \geq &
\frac{1}{1+\frac{c(p)}{\lambda^2nf(2p)}}\left(1-
    \frac{4}{\pi}\sqrt{\frac{5c(p)}{\left(1-\lambda\right)\epsilon n
    f(2p)}}\right)
\sum_{i=1}^p \left(\frac{1}{2} \log \frac{n (1-\lambda)f(2i)}{2c(p)\pi\mathe} +\log
  \pi-\epsilon  \right) \, 
  \\
  && \text{using monotonicity assumption on $f$.} 
\end{eqnarray*}

\end{proof}

\section{Examples of  envelope classes}
\label{sec:ex}
Theorems~\ref{prop:envfinite},~\ref{prop:upperbound} and~\ref{prop:desperate} assert that the summability of the
envelope defining a class of memoryless sources characterizes the
(strong) universal compressibility of that class. However, it is not
easy to figure out whether the bounds provided by the last two
theorems are close to each other or not.
In this Section, we investigate the case of 
envelopes which decline either like power laws or exponentially fast. 
In both cases, upper-bounds on minimax regret will follow directly
from Theorem~\ref{prop:upperbound} and a straightforward
optimization. Specific lower bounds on minimax redundancies are
derived by mimicking the proof of Theorem~\ref{prop:desperate}, either
faithfully as in the case of  exponential envelopes or by developing
an alternative prior as in the case of power-law envelopes. 



\subsection{Power-law envelope classes}
\label{sec:power-law-envelope}

Let us first agree on the classical notation: 
\begin{math}
  \zeta(\alpha)=\sum_{k\geq 1} \frac{1}{k^\alpha}\, ,\text{ for
  }\alpha>1\, .
\end{math}
\begin{thm}\label{th:powerlaw}
Let $\alpha$ denote a real number larger than $1,$ and $C$ be such
that $C \zeta(\alpha)\geq 2^\alpha. $
The source class $\Lambda_{C\, \cdot^{-\alpha}}$ is the envelope class
associated with the decreasing function
\begin{math}
  f_{\alpha, C}:x\mapsto1\wedge \frac{C}{x^\alpha}
\end{math}
for $C>1$ and  $\alpha>1 .$ \\
Then:
\begin{enumerate}
\item  
\[
n^{1/\alpha}\, A(\alpha) \, \log \left\lfloor ({C}{\zeta(\alpha)})
^{1/\alpha} 
\right\rfloor \leq R^+(\Lambda_{C\, \cdot^{-\alpha}}^n)
\]
where 
$$
A(\alpha) = \frac{1}{\alpha}  \int_1^{\infty} \frac{1}{u^{1
- 1 / \alpha}}  \left( 1 - \text{e}^{ - {1}/({\zeta (\alpha) u)} }
\right) \mathd u
\, . $$
\item 
\[
 R^*(\Lambda_{C\, \cdot^{-\alpha}}^n)\leq
 \left(\frac{2Cn}{\alpha-1}\right)^{1/\alpha}\left(\log
 n\right)^{1-1/\alpha}+O(1)\, .
\]
\end{enumerate}

\end{thm}
\begin{rem}
The gap between the lower-bound  and the upper-bound is of order $\left(\log n\right)^{1-\frac{1}{\alpha}}.$ 
We are not in a position to claim that one of the two bounds is tight, let alone which one is tight.
Note however that as $\alpha\to\infty$ and $C=H^\alpha$, class
$\Lambda_{C\, \cdot^{-\alpha}}$ converges to the class of memoryless
sources on alphabet $\{1, \ldots, H\}$ for which the minimax
regret  is $\frac{H-1}{2}\log n$.
This is (up to a factor 2) what we obtain by taking the limits in our upper-bound of $R^*(\Lambda_{C\, \cdot^{-\alpha}}^n)$.
On the other side, the limit of our lower-bound when $\alpha$ goes to $1$ is infinite, 
which is also satisfying since it agrees with Theorem~\ref{prop:envfinite}.
\end{rem}

\begin{rem}
In contrast with various  lower bounds derived using a similar
methodology, the proof given here 
relies on a single prior probability distribution on the parameter
space and works for all values of $n.$

Note that the lower bound that can be derived from
Theorem~\ref{prop:desperate} is of the same order of magnitude
$O(n^{1/\alpha})$ as the lower bound stated here (see Appendix \ref{sec:lower:bound:envelope}).
 The proof given here
is completely elementary and does not rely on the subtle computations
described in~\cite{xie:barron:1997}.
\end{rem}

\begin{proof}
For the upper-bound on minimax regret, note that
$$
\bar{F}_{\alpha, C}(u) = \sum_{k>u} 1\wedge \frac{C}{k^\alpha} \leq 
\frac{C}{(\alpha-1) \, u^{\alpha-1}}.
$$
Hence, choosing $u_n=\left(\frac{2Cn}{(\alpha-1)\log
    n}\right)^{\frac{1}{\alpha}},$ resorting to
Theorem~\ref{prop:upperbound},  we get:
$$R^*(\Lambda_{C\, \cdot^{-\alpha}}^n) \leq
\left(\frac{2Cn}{\alpha-1}\right)^{1/\alpha}\left(\log
n\right)^{1-1/\alpha}+O(1).
$$ 


Let us now turn to the lower bound.
We first define a finite set $\Theta$ of parameters such that
$P^n_{\boldsymbol{\theta}} \in \Lambda^n_{\alpha,C}$ for any
$\boldsymbol{\theta}\in \Theta$
and then we use the mutual information lower bound on redundancy. 

Let $m$ be a positive integer such that $ m^{\alpha}\leq
C\zeta(\alpha)\,.$ 

The set $\{ P_{\boldsymbol{\theta}}, \boldsymbol{\theta}\in \Theta\}$
consists
 of memoryless sources over the finite alphabet $\Np.$
Each parameter $\boldsymbol{\theta}$ is a sequence of integers $\boldsymbol{\theta}= (\theta_1 ,\theta_2,\ldots,)$. 
We take a prior distribution on $\Theta$ such that $(\theta_k)_k$ is a sequence of independent identically distributed random variables with uniform distribution on $\{1,\ldots,m\}$. 
For any such $\boldsymbol{\theta},$  $P^1_{\boldsymbol{\theta}}$ is a probability
distribution on $\Np$ with
support $\cup_{k\geq 1}\{ (k-1)m +\theta_k\},$ namely:
\begin{equation}
  P_{\boldsymbol{\theta}}((k-1)m +\theta_k) = \frac{1}{\zeta (\alpha) k^\alpha}
  = \frac{m^{\alpha}}{\zeta (\alpha)}\cdot\frac{1}{(k\,m)^\alpha}\quad \text{for }k\geq 1.
\end{equation}
The condition  $ m^{\alpha}\leq
C\zeta(\alpha)$ 
 ensures that $P^1_{\boldsymbol{\theta}}\in \Lambda^{1}_{\alpha,C}$. 


Now, the mutual information between parameter $\boldsymbol{\theta}$
and source output $X_{1:n}$ is
\begin{eqnarray*}
I\left(\boldsymbol{\theta}, X_{1:n}\right) &=& \sum_{k\geq 1} I\left(\theta_k, X_{1:n}\right) 
\end{eqnarray*}

Let $N_k\left(\mathbf{x}\right)=1$ if there exists some index
$i\in\{1,\ldots,n\}$ such that $\mathbf{x}_i\in[(k-1)m+1, km]$, and $0$
otherwise. Note that the distribution of $N_k$ does not depend on the value of $\boldsymbol{\theta}.$ Thus we can write:

\begin{eqnarray*}
 I\left(\theta_k, X_{1:n}\right) & = &  I\left(\theta_k, X_{1:n} |
   N_k
=0\right)
 \PROB\left(N_k 
=0\right) +  I\left(\theta_k, X_{1:n} |  N_k
=1\right) \PROB\left(N_k
=1\right).
\end{eqnarray*}

But,  conditionally on  $N_k=0,$   $\theta_k$ and $X_{1:n}$ are independent.
Moreover, conditionally on   $N_k 
=1$ we have 
\begin{displaymath}
I\left(\theta_k, X_{1:n}\mid N_k=1\right) =  \E\left[\log
  \frac{\PROB(\theta_k=j | X_{1:n})}{\PROB(\theta_k=j)} |
  N_k
=1\right]= \log m.
\end{displaymath}

Hence,
\begin{eqnarray*}
I\left(\boldsymbol{\theta}, X_{1:n}\right) & = &\sum_{k\geq 1}
\PROB(N_k 
=1) \log m \\
 & = & \mathbbm{E}_{P_{\boldsymbol{\theta}}}\left[Z_n \right]  \log m,
\end{eqnarray*}
where $Z_n(\mathbf{x})$ denotes the number of distinct symbols in string $\mathbf{x}$ (note that its distribution does not depend on the value of $\boldsymbol{\theta}.$)
As
\begin{math}
Z_n = \sum_{k\geq 1} 1_{N_k(x)=1},
\end{math} 
the expectation translates into a sum
$$
\mathbbm{E}_{\theta}\left[Z_n \right]=\sum_{k=1}^{\infty}\left(
1-\left(1-\frac{1}{\zeta (\alpha) k^\alpha} \right)^{n}\right) 
$$
which  leads to:
$$
R^+(\Lambda^n_{\alpha,C}) \geq   \left(\sum_{k=1}^{\infty}\left(
1-\left(1-\frac{1}{\zeta (\alpha) k^\alpha} \right)^{n}\right)\right) \times \log m \,. 
$$
Now: 
\begin{eqnarray*}
\lefteqn{\sum_{k=1}^{\infty}\left(1-\left(1-\frac{1}{\zeta (\alpha) k^\alpha} \right)^{n}\right)}\\
&\geq &\sum_{k=1}^{\infty}\left(1-\exp\left(-\frac{n  }{\zeta(\alpha)\,k^{\alpha}} \right)\right)\\
 & &\text{as $1-x \leq \exp(-x) $} \\
&\geq& \int_{1}^{\infty}\left(1-\exp\left(-\frac{n}{\zeta(\alpha) x^{\alpha}} \right)\right)\text{d}x\\
&\geq &\frac{n^{\frac{1}{\alpha}}}{\alpha}\int_{1}^{\infty}\frac{1}{u^{1-\frac{1}{\alpha}}}\left(
1-\exp\left(-\frac{1}{\zeta(\alpha)u}\right)\right)\text{d}u\, 
.
\end{eqnarray*}
\end{proof}
In order to optimize the bound we choose the largest possible $m$ which is $m=\lfloor (C
\zeta(\alpha))^{1/\alpha}\rfloor\, .$\\
For an alternative derivation of a similar lower-bound using
Theorem~\ref{prop:desperate}, see Appendix
\ref{sec:lower:bound:envelope}.

\subsection{Exponential envelope classes}
\label{sec:expon-envel-class}

Theorems~\ref{prop:upperbound} and~\ref{prop:desperate} provide almost
matching bounds on the minimax redundancy for source classes defined
by exponentially vanishing envelopes.

\begin{thm}\label{th:expo}
Let $C$  and $\alpha$ denote  positive real numbers satisfying $C>e^{2\alpha}.$ 
The class $\Lambda_{C e^{-\alpha\cdot}}$ is the envelope class  associated with function $f_\alpha:x\mapsto 1\wedge C \text{e}^{-\alpha x}.$
Then 
\begin{displaymath}
\frac{1}{8\alpha}\log^2 n \left(1-o(1)\right)\leq R^+(\Lambda_{C e^{-\alpha\cdot}}^n) \leq R^*(\Lambda_{C e^{-\alpha\cdot}}^n) \leq \frac{1}{2\alpha} \log^2 n+O(1)
\end{displaymath}
\end{thm}
 
\begin{proof}
For the upper-bound, note that  
$$
\bar{F}_\alpha (u) = \sum_{k>u} 1\wedge  {C}\text{e}^{-\alpha k} \leq \frac{C}{1-\text{e}^{-\alpha}} \,  \text{e}^{-\alpha (u+1)} \, .
$$
Hence, by choosing the optimal value $u_n=\frac{1}{\alpha}\log n$ in Theorem~\ref{prop:upperbound} we get:
$$R^*(\Lambda_{C e^{-\alpha\cdot}}^n) \leq
\frac{1}{2\alpha} \log^2 n+O(1).
$$ 

We will now prove the lower bound using Theorem~\ref{prop:desperate}. The constraint
$C>\text{e}^{2\alpha}$ warrants that the sequence $c(p)=\sum_{k=1}^p
f(2k)\geq  C\text{e}^{-2\alpha}\frac{1-\text{e}^{-2\alpha p}}{1-\text{e}^{-2\alpha}}$ is larger than $1$ for all $p.$ 

If we choose $p=\left\lfloor\frac{1}{2\alpha}\left(\log n-\log\log
n\right)\right\rfloor$,  then 
$nf(2p)>Cne^{-\log n+\log \log n-2\alpha}$ goes to infinity with $n$.
For $\epsilon=\lambda=\frac{1}{2}$, we get
$C(p,n,\lambda,\epsilon) = 1-o(1)$.
Besides,
\begin{eqnarray*}
\sum_{i=1}^p\left( \frac{1}{2} \log\frac{n(1-\lambda)C\pi e^{-2\alpha
    i}}{2c(p) e}-\epsilon\right) & = &\frac{p}{2}\left(\log n + \log
    \frac{(1-\lambda)C\pi}{2c(p)e}
    -2\epsilon\right)-\alpha\sum_{i=1}^p i\\
 & = & \left(\frac{1}{4\alpha}\log^2 n -
    \frac{\alpha}{2}\frac{1}{4\alpha^2} \log^2 n\right)
    \left(1+o(1)\right)\\
&=&\frac{1}{8\alpha}\log^2 n\left(1+o(1)\right).
\end{eqnarray*}

\end{proof}

\def\1{\mathbbm{1}}
\def\dd{\;d}

\section{A censoring code for envelope classes}
\label{sec:algo}

The proof of Theorem~\ref{prop:upperbound}
suggests  to handle
separately small and large (allegedly infrequent) symbols. Such an algorithm should
perform quite well as soon as the tail behavior of the envelope
provides an adequate description of the sources in the class.  The
coding algorithm suggested by the proof of Theorem~\ref{prop:upperbound},
which are based on the Shtarkov \textsc{nml} coder, are not
computationally attractive. 
The design of the next algorithm (\texttt{CensoringCode}) is again guided by
the proof of Theorem~\ref{prop:upperbound}: it is
parameterized by  a sequence of cutoffs $(K_i)_{i\in \mathbbm{N}}$ and
handles the $i^{\text{th}}$ symbol of the sequence to be encoded 
differently  according to whether it is  smaller or larger than
cutoff $K_i,$  in the latter situation, the symbol is said to be
censored. The \texttt{CensoringCode} algorithm  uses Elias penultimate
code~\citep{MR0373753} to encode censored symbols and
Krichevsky-Trofimov mixtures~\citep{krichevsky:trofimov:1981}  to encode the sequence of non-censored
symbols padded with markers (zeros) to witness acts of censorship. 
The performance of this algorithm is evaluated on the power-law envelope class
$\Lambda_{C\cdot^{-\alpha}}$,  already investigated in
Section~\ref{sec:ex}. 
In this section, the parameters 
$\alpha$  and $C$  are assumed to be known. 

Let us first describe the algorithm more precisely. Given a non-decreasing sequence
of cutoffs $(K_i)_{i\leq n},$  a string $\mathbf{x}$ from $\Np^n$
defines two strings $\tilde{\mathbf{x}}$ and $\check{\mathbf{x}}$ in
the following way. The $i^{\mathrm{th}}$ symbol $\mathbf{x}_i$ of $\mathbf{x}$ is censored if $\mathbf{x}_i>k_i.$ String $\tilde{\mathbf{x}}$  has length $n$ and
belongs to $\prod_{i=1}^n \mathcal{X}_i,$ where
$\mathcal{X}_i=\{0,\ldots K_i\}$:
\begin{displaymath}
  \tilde{\mathbf{x}}_i=   \left\{\begin{array}{cl}\mathbf{x}_i & \hbox{if } \mathbf{x}_i\leq
K_i \\ 0 & \hbox{otherwise (the symbol is censored).}\end{array}\right.
\end{displaymath}
Symbol $0$ serves as an escape symbol. Meanwhile, string $\check{\mathbf{x}}$ is the subsequence of censored symbols, that is $(\mathbf{x}_i)_{\mathbf{x}_i>K_i,i\leq n}.$

The algorithm encodes $\mathbf{x}$ as a pair of binary strings
$\mathtt{C1}$ and $\mathtt{C2}.$  The first one ($\mathtt{C1}$) is
obtained by applying Elias penultimate code to each symbol from
$\check{\mathbf{x}},$ that is to each censored symbol.
 The second string ($\mathtt{C2}$) is built
by applying arithmetic coding to $\tilde{\mathbf{x}}$ using
side-information from $\check{\mathbf{x}}.$ Decoding $\mathtt{C2}$ can
be sequentially carried out  using information obtained from decoding $\mathtt{C1}.$ 

In order to describe the coding probability used to encode
$\tilde{\mathbf{x}},$
we need a few more counters. For $j>0,$ let $n^j_i$ be the number of
occurrences
of symbol $j$ in $\mathbf{x}_{1:i}$ and let $n^0_i$ be the number of
symbols larger than $K_{i+1}$ in $\mathbf{x}_{1:i}$ (note that this
not larger than the number of censored symbols in $\mathbf{x}_{1:i},$
and  that the counters $n^\cdot_i$ can be recovered from
$\tilde{\mathbf{x}}_{1:i}$ and $\check{\mathbf{x}}$).
The conditional coding probability over alphabet
$\mathcal{X}_{i+1}=\{0,\ldots,K_{i+1}\}$
given $\tilde{\mathbf{x}}_{1:i}$ and $\check{\mathbf{x}}$ is derived from 
the Krichevsky-Trofimov mixture over $\mathcal{X}_{i+1}.$ It is the
posterior
distribution corresponding to Jeffrey's prior on the
$1+K_{i+1}$-dimensional probability simplex  and counts $n^j_i$ for
$j$  running from $0$ to $K_{i+1}$:
\begin{displaymath}
  Q(\tilde{X}_{i+1}= j \mid
  \tilde{X}_{1:i}=\tilde{\mathbf{x}}_{1:i},\check{X} = \check{\mathbf{x}}) =
  \frac{n^j_i+\frac{1}{2}}{i+ \frac{K_{i+1}+1}{2}} \, . 
\end{displaymath}
The length of $\mathtt{C2}(\mathbf{x})$  is (up to a quantity smaller
than $1$) given by
\begin{displaymath}
  -\sum_{i=0}^{n-1}\log Q(\tilde{\mathbf{x}}_{i+1} \mid
  \mathbf{x}_{1:i},\check{\mathbf{x}}) = - \log Q(\tilde{\mathbf{x}}\mid
  \check{\mathbf{x}})\, .  
\end{displaymath}
The following description of the coding probability will prove
useful when upper-bounding redundancy. 
For $1\leq j\leq K_n,$ let $s^j$ be the number of censored occurrences of symbol
$j.$ Let $n^j$ serve as a shorthand for $n^j_n.$
Let  $i(j)$ be $n$ if $K_n<j$ or the largest integer $i$ such that $K_{i}$ is smaller
than $j$, then $s_j = n^j_{i(j)}$. The following holds 
\begin{displaymath}
  Q(\tilde{\mathbf{x}} \mid \check{\mathbf{x}}) =  \left(\prod_{j=1}^{K_n}
    \frac{\Gamma(n^j+1/2)}{\Gamma(s^j+1/2)}\right) 
\left(\prod_{i: \tilde{\mathbf{x}}_i=0} (n^0_{i-1}+1/2)\right)  \left(\prod_{i=0}^{n-1}\frac{1}{i+\frac{K_{i
+1
}+1}{2}}\right)
\end{displaymath}
Note that the sequence $(n^0_i)_{i\leq n}$ is not necessarily
non-decreasing. 

A technical description of algorithm \texttt{CensoringCode} is given below. 
The procedure  \texttt{EliasCode}
takes as input an integer $j$ and outputs a binary encoding  of $j$
using exactly $\ell(j)$ bits where $\ell$ is defined by:
\begin{math}
\ell(j)=\left\lfloor \log j +2\log\left(1+\log j\right)+1\right\rfloor\, .
\end{math}
The procedure \texttt{ArithCode}  builds on the arithmetic coding
methodology~\citep{rissanen:langdon:1984}. It is enough
to  remember that an arithmetic coder takes advantage of the fact that
a coding probability $Q$ is completely defined by the  sequence of
conditional distributions of the $i+1$th symbol given the past up to
time $i.$  

The proof of the upper-bound in Theorem
\ref{th:powerlaw}, prompts us to choose
 $K_i=\lambda i^{\frac{1}{\alpha}}$ , it will become clear afterward
 that
a reasonable choice is  $\lambda = \left(\frac{4C}{\alpha-1}\right)^{\frac{1}{\alpha}}$.

\begin{algorithm}
\caption{CensoringCode}
\begin{algorithmic}
\label{algo:censoring}
\STATE $K \leftarrow 0$
\STATE $\text{\em counts} \leftarrow [1/2, 1/2,\ldots ]$
\FOR{$i$ from $1$ to $n$}{
\STATE $\text{\em cutoff} \leftarrow \left\lfloor\left( 4\frac{C i
  }{\alpha-1}\right)^{1/\alpha}\right\rfloor$
\IF{$\text{\em cutoff}> K$}{
\FOR{$j\leftarrow K+1$ to $\text{\em cutoff}$}{\STATE $\text{\em
    counts}[0]\leftarrow \text{\em counts}[0]-\text{\em counts}[j]+1/2
  $}
\ENDFOR
\STATE $K \leftarrow \text{\em cutoff}$
}
\ENDIF
\IF{$x[i]\leq \text{\em cutoff}$}{
  \STATE 
$\text{ArithCode}(x[i],\text{\em     counts}[0:\text{\em  cutoff}]) $
}
\ELSE{
\STATE 
$\text{ArithCode}(0,\text{\em
    counts}[0:\text{\em  cutoff}]) $
\STATE \texttt{C1}$ \leftarrow $\texttt{C1}$\cdot \text{EliasCode}(x[i]) $
\STATE $\text{\em
    counts}[0] \leftarrow \text{\em
    counts}[0] +1$}
\ENDIF
\STATE $\text{\em
    counts}[x[i]] \leftarrow \text{\em
    counts}[x[i]] +1$
}
\ENDFOR
\STATE \texttt{C2}$ \leftarrow \text{ArithCode}()$
\STATE $\texttt{C}_1 \cdot \texttt{C}_2$
\end{algorithmic}
\end{algorithm}

\begin{thm}\label{th:redcc}
	Let $C$ and $\alpha$  be positive reals.
Let the sequence of cutoffs $(K_i)_{i\leq n}$ be  given by
\begin{displaymath}
  K_i =\Big\lfloor \left(\frac{4C i}{\alpha-1}\right)^{1/\alpha} \Big\rfloor
  \, . 
\end{displaymath}
The expected redundancy of procedure \texttt{CensoringCode} on the envelope class
$\Lambda_{C \cdot^{-\alpha}}$ is not larger than 
$$\left(\frac{4Cn}{\alpha-1}\right)^{\frac{1}{\alpha}}\log n\left(1+o(1)\right).$$
\end{thm}
\begin{rem}
The redundancy upper-bound in this Theorem is within a factor $\log n$
from the lower bound $O(n^{1/\alpha})$ from Theorem~\ref{th:powerlaw} .
\end{rem}

The proof of the Theorem builds on the next two lemmas. 
The first lemma compares the length of $\mathtt{C2}(\mathbf{x})$ with
a tractable quantity. The second lemma upper-bounds the average length
of $\mathtt{C1}(\mathbf{x})$ by a quantity which is of the same order
of magnitude as the upper-bound on redundancy we are looking for.

We need a few more definitions.
Let $\mathbf{y}$ be the string of length $n$ over alphabet
$\mathcal{X}_n$ defined by:
$$\mathbf{y}_i = \left\{\begin{array}{cl}\mathbf{x}_i & \hbox{if } \mathbf{x}_i\leq
K_n ; \\ 0 &
\hbox{else}.\end{array}\right.$$
For $0\leq j\leq K_n$, note that the previously defined shorthand
$n^j$ is the number of occurrences of symbol $j$ in $\mathbf{y}.$ 
The string $\mathbf{y}$ is obtained from $\mathbf{x}$ in the same way
as $\tilde{\mathbf{x}}$ using the constant cutoff $K_n.$

Let  $m^*_n$ be the Krichevsky-Trofimov mixture over alphabet $\{0,\ldots,K_n\}$:
\begin{displaymath}
  m^*_n(\mathbf{y}) = \left( \prod_{j=0}^{K_n}
  \frac{\Gamma(n^j+\frac{1}{2})}{\Gamma(1/2)} \right) 
  \frac{\Gamma(\frac{K_n+1}{2})}{\Gamma\left(n+\frac{K_n+1}{2} \right)}
  \,. 
\end{displaymath}
String $\tilde{\mathbf{x}}$ seems easier to encode than
$\mathbf{y}$ since it is possible to recover $\tilde{\mathbf{x}}$ from
$\mathbf{y}.$ This observation does not however warrant automatically  that the length of
$\texttt{C2}(\mathbf{x})$ is not significantly larger than  any
reasonable codeword length for $\mathbf{y}.$ Such a guarantee is
provided by the following lemma. 

\begin{lem}\label{nml:lem1}
For every string $\mathbf{x}\in\Np^n$, the length of the  \texttt{C2}$\left(\mathbf{x}\right)$ is not larger than
$-\log m^*_n(\mathbf{y})$.
\end{lem}
\begin{proof}[Lemma \ref{nml:lem1}]
Let $s^0$ be the number of occurrences of $0$ in $\mathbf{y},$ that is
the number of symbols in $\mathbf{x}$ that are larger than $K_n$. Let 
\begin{displaymath}
T_0 = \prod_{i=1, \tilde{\mathbf{x}}_i=0}^n  (n^0_{i-1}+1/2) \,   . 
\end{displaymath}
Then, the following holds:  
\begin{eqnarray*}
T_0 
& \stackrel{(a)}{=} &  
\left( \prod_{i=1, \mathbf{y}_i=0}^n  (n^0_{i-1}+1/2)\right)
\prod_{j=1}^{K_n} \left( \prod_{i=1, \mathbf{y}_i=j}^{i(j)}
  (n^0_{i-1}+1/2)\right)
\\
& \stackrel{(b)}{\geq }  & 
\left( \frac{\Gamma(s^0+1/2)}{\Gamma(1/2)} \right)
\prod_{j=1}^{K_n} \left(\frac{  \Gamma(s^j+1/2)}{\Gamma(1/2)}  \right)
\,,
\end{eqnarray*}
where $(a)$ follows from the fact symbol $\mathbf{x}_i$ is censored
either because $\mathbf{x}_i>K_n$ (that is $\mathbf{y}_i=0$) or
because $\mathbf{x}_i=j\leq K_n$ and $i\leq i(j)$; $(b)$ follows from
the fact that for each $i\leq n$ such that ${\mathbf{y}}_i=0,$
$n^0_{i-1}\geq \sum_{i'<i} \mathbf{1}_{\mathbf{x}_{i'}>K_n}$  while
for
each $j, 0<j\leq K_n,$
for each $i\leq i(j),$  $n^0_{i-1}\geq n^j_{i-1}.$

From the last inequality, it follows that 
\begin{eqnarray*}
  Q(\tilde{\mathbf{x}}\mid \check{\mathbf{x}}) 
& \geq & \left(\prod_{j=1}^{K_n}
    \frac{\Gamma(n^j+1/2)}{\Gamma(s^j+1/2)}\right) 
\frac{\Gamma(s^0+1/2)}{\Gamma(1/2)} \prod_{j=1}^{K_n}
  \frac{\Gamma(s^j+1/2)}{\Gamma(1/2)}
\left(\prod_{i=0}^{n-1}\frac{1}{i+\frac{K_{i
+1
}+1}{2}}\right)
\\
& \geq & 
\left(\prod_{j=1}^{K_n}
    \frac{\Gamma(n^j+1/2)}{\Gamma(s^j+1/2)}\right) 
\frac{\Gamma(s^0+1/2)}{\Gamma(1/2)} \prod_{j=1}^{K_n}
  \frac{\Gamma(s^j+1/2)}{\Gamma(1/2)}
\left(\prod_{i=0}^{n-1}\frac{1}{i+\frac{K_n+1}{2}}\right)
\\
& = & m^*_n(\mathbf{y}) \, ,
\end{eqnarray*}
 where the last  inequality holds since $\left(K_i\right)_i$ is a
 non-decreasing sequence. 
\end{proof}


The next lemma shows that the expected length of $\texttt{C1}(X_{1:n})$
is not larger than the upper-bound we are looking for.  
\begin{lem}\label{nml:lem2}
For every source $P\in\Lambda_{C\cdot^{-\alpha}}$, the expected length
of the encoding of the censored symbols ($\texttt{C1}(X_{1:n})$)  satisfies: 
$$\E_P\big[\left|\hbox{\texttt{C1}}\left(X_{1:n}\right)\right|\big]\leq
\frac{2C}{\left(\alpha-1\right)\lambda^{\alpha-1}}n^{\frac{1}{\alpha}}\log
n\left(1+o(1)\right).$$ 
\end{lem}

\vspace{.5cm}
\begin{proof}[Lemma \ref{nml:lem2}]
 Let $1\leq a<b$ and $\beta>0$ but $\beta\neq 1$. Recall that we use binary logarithms and note that:
 \begin{eqnarray}
 \int_a^b \frac{1}{x^\beta}\mathrm{d}x& =& \left[\frac{1}{\left(1-\beta\right)x^{\beta-1}}\right]_a^b,\label{I1}\\
 \int_a^b \frac{\log x}{x^\beta}\mathrm{d}x 
 &=&  \left[\frac{\log x -\frac{\log
       e}{1-\beta}}{\left(1-\beta\right)x^{\beta-1}}\right]_a^b \,. \label{I2}
 \end{eqnarray}

The expected length of the Elias encoding of censored symbols ($\mathtt{C1}(X_1^n)$)
is:
\begin{eqnarray*}
\E\big[|\hbox{\texttt{C1}}\left(X_1^n\right)|\big] &=& \E\left[\sum_{j=1}^n\ell(X_j)\1_{X_j>K_j}\right] \\
& = & \sum_{j=1}^n \sum_{x=K_j+1}^\infty \ell(x)P(x)\\
& \leq & \sum_{j=1}^n\sum_{x=K_j+1}^\infty \ell(x) \frac{C}{x^\alpha}\\
& \leq & C\sum_{j=1}^n\sum_{x=K_j+1}^\infty \frac{\log(x)+2\log\left(1+\log x\right)+1}{x^\alpha}.
\end{eqnarray*}
Note that for $x\geq 2^7$
\begin{displaymath}
  \log x \geq 2 \log(1+ \log x) +1, 
\end{displaymath}
so that the last sum is upper-bounded by 
\begin{displaymath}
   C \sum_{j=1}^n \sum_{x=K_j+1}^\infty 2 \frac{\log x }{x^\alpha} + C
   \sum_{j: K_j< 2^7} \sum_{x=1}^{2^7} \frac{2 \log(1+\log x)+1-\log
     x}{x^\alpha} \, .
\end{displaymath}
Using the expressions for the integrals above, we get: 
\begin{eqnarray*}
\sum_{x=K_j+1}^\infty \frac{\log x}{x^\alpha} &\leq& \int_{K_j}^\infty \frac{\log x}{x^\alpha}\mathrm{d}x\\
 & \leq & \frac{\log K_j+\frac{\log e}{\alpha-1}}{\left(\alpha-1\right)K_j^{\alpha-1}}.
\end{eqnarray*}
 
Thus, as $K_j=\lambda j^{1/\alpha}$, let us denote by $D_\lambda$ the
expression 
\begin{displaymath}
  \sum_{j: j < (2^7/\lambda)^\alpha} \sum_{x=1}^{2^7} \frac{2 \log(1+\log x)+1-\log
     x}{x^\alpha}\, . 
\end{displaymath}
Now, we substitute $\beta$ by
$1-\frac{1}{\alpha}$ in Equations (\ref{I1})
and  (\ref{I2}) 
to obtain:
\begin{eqnarray*}
\E\big[|\hbox{\texttt{C1}}\left(X_1^n\right)|\big] &\leq & 
C D_\lambda + 
\frac{2 C}{\alpha-1} \sum_{j=1}^n \frac{\log K_j+\frac{\log e}{\alpha-1}}{K_j^{\alpha-1}}\\
& \leq & C D_\lambda +   \frac{2 C}{\alpha-1}  \sum_{j=1}^n
\frac{\frac{1}{\alpha}\log j + \log \lambda
+\frac{\log e}{\alpha-1}}{\lambda^{\alpha-1}j^{1-\frac{1}{\alpha}}}\\
& \leq & C D_\lambda+\frac{2C}{\alpha(\alpha-1)\lambda^{\alpha-1}}\left(C+\int_{x=2}^{n+1} \frac{\left(\log x + \alpha\log \lambda+\frac{\alpha\log e}{\alpha-1}\right)}{x^{1-\frac{1}{\alpha}}} \mathrm{d}x\right)\\
& =& \frac{2C}{\left(\alpha-1\right)\lambda^{\alpha-1}}n^{\frac{1}{\alpha}}\log n\left(1+o(1)\right).
\end{eqnarray*}
\end{proof}

We may now complete the proof of Theorem~\ref{th:redcc}.

\begin{proof}
  Remember that $\mathcal{X}_n=\left\{0,\ldots,K_n\right\}$. If $p$ is
  a probability mass function over alphabet $\mathcal{X}$, let
  $p^{\otimes n}$ be the probability mass function over
  $\mathcal{X}^n$ defined by 
$p^{\otimes n}(\mathbf{x}) = \prod_{i=1}^n p(\mathbf{x}_i).$
 Note
that for every string $\mathbf{x}\in\Np^n$, 
\begin{displaymath}
\max_{p\in\mathfrak{M_1}\left(\mathcal{X}_n\right)}
  p^{\otimes n} (\mathbf{y}) \geq
  \max_{p\in\mathfrak{M_1}\left(\Np\right)}
  p^{\otimes n} (\mathbf{x}) 
 \geq   \max_{P\in \Lambda_{C\cdot^{-\alpha}}}  P^n(\mathbf{x}) =
 \hat{p}(\mathbf{x})\, .
\end{displaymath}
Together with Lemma \ref{nml:lem1} and the bounds on the redundancy of the
Krichevsky-Trofimov mixture \citep[See][]{krichevsky:trofimov:1981}, this implies: 
\begin{eqnarray*}
|\hbox{\texttt{C2}}\left(\mathbf{x}\right)|
& \leq & -\log \hat {p}\left(\mathbf{x}\right) + \frac{K_n}{2} \log n +O(1).
\end{eqnarray*}

Let $L(\mathbf{x})$ be the length of the code produced by algorithm
\texttt{CensoringCode} on the input string $\mathbf{x}$, then 
\begin{eqnarray*}
\lefteqn{\sup_{P\in\Lambda_{C\cdot^{-\alpha}}}\E_P\left[L(X_{1:n}) - \log 1/P^n(X_{1:n})\right]}\\
& \leq & \sup_{P\in\Lambda_{C\cdot^{-\alpha}}} \E_P\left[L(X_{1:n}) - \log 1/\hat{p}(X_{1:n})\right]\\
& \leq &   \sup_{P\in\Lambda_{C\cdot^{-\alpha}}^n} \E_P\left[
  \left|\hbox{\texttt{C2}}\left(X_{1:n}\right)\right| +  \log
  \hat{p}\left(X_{1:n}\right)
  +\left|\hbox{\texttt{C1}}\left(X_{1:n}\right)\right| \right] \\
& \leq & \sup_{\mathbf{x}} \left(\left|\hbox{\texttt{C2}}\left(\mathbf{x}\right)\right| +  \log
  \hat{p}\left(\mathbf{x}\right) \right) +\sup_{P\in\Lambda_{C\cdot^{-\alpha}}^n} \E_P\left[\left|\hbox{\texttt{C1}}\left(X_{1:n}\right)\right| \right]
\\
& \leq & \frac{\lambda n^{\frac{1}{\alpha}}}{2}\log n +
\frac{2C}{\left(\alpha-1\right)\lambda^{\alpha-1}}n^{\frac{1}{\alpha}}\log n\left(1+o(1)\right).
\end{eqnarray*}

The optimal value is $\lambda=\left(\frac{4C}{\alpha-1}\right)^{\frac{1}{\alpha}}$, for which we get:
$$R^+(Q^n, \Lambda_{C\cdot^{-\alpha}}^n)\leq \left(\frac{4Cn}{\alpha-1}\right)^{\frac{1}{\alpha}}\log n \left(1+o(1)\right).$$
\end{proof}

\section{Adaptive algorithms}
\label{sec:an-adaptive-version}

The performance of \texttt{CensoringCode} depends on the fit of the
cutoffs sequence to the tail behavior of the envelope. From the proof
of Theorem~\ref{th:redcc}, it should be clear that if
\texttt{CensoringCode} is fed with a source which marginal is
light-tailed, it will be unable to take advantage of this, and will
suffer from excessive redundancy. 

In this section, a sequence $(Q^n)_{n}$  of coding probabilities is said
to be \emph{approximately asymptotically adaptive} with respect to a collection
$(\Lambda_m)_{m\in \mathcal{M}}$ of  source classes if for each $P\in
\cup_{m\in \mathcal{M}} \Lambda_m,$ for each $\Lambda_m$  such that
$P\in \Lambda_m$: 
\begin{displaymath}
  D(P^n,Q^n) /R^+(\Lambda_m^n) \in O(\log n) \, .
\end{displaymath}
Such a definition makes sense, since we are considering massive source
classes which minimax redundancies are large but still  $o(\frac{n}{\log n})$.
If each class $\Lambda_m$ admits a non-trivial redundancy rate such that
$R^+(\Lambda^n_m)=o(\frac{n}{\log n})$, the existence of an approximately asymptotically adaptive sequence of  coding probabilities means that $\cup_m \Lambda_m$ is feebly universal (see the Introduction for a definition).

\subsection{Pattern coding}

First, the use of \emph{pattern coding} \cite{MR2095850,MR2234457} leads to
an almost minimax adaptive procedure for small values of $\alpha$,
that is heavy-tailed distributions.
Let us introduce the notion of pattern using the example of string
$\mathbf{x}=\hbox{``abracadabra''},$ which is made of $n=11$ characters. 
The information it conveys can be separated in two blocks:
\begin{enumerate}
\item a \emph{dictionary} $\Delta=\Delta(\mathbf{x})$: the sequence of
  distinct symbols occurring  in $\mathbf{x}$  in order of appearance
  (in the example, $\Delta=(a,b,r,c,d)$).
\item a \emph{pattern} $\psi=\psi(\mathbf{x})$ where
$\psi_i$ is the rank of $\mathbf{x}_i$ in the dictionary $\Delta$ (here, $\psi=1231415123$).
\end{enumerate}

Now, consider the algorithm coding message $\mathbf{x}$ by transmitting successively:
\begin{enumerate}
\item the dictionary $\Delta_n =  \Delta\left(\mathbf{x}\right)$ (by concatenating the Elias codes for
  successive symbols); 
\item and the pattern $\Psi_n=\psi\left(\mathbf{x}\right)$, using a minimax
  procedure for coding patterns as suggested by \cite{MR2095850} or
  \cite{MR2234457}.
Henceforth, the latter procedure is called pattern coding. 
\end{enumerate}
\begin{thm} \label{th:patterncoding}
Let $Q^n$ denote the coding probability associated with 
  the coding algorithm which consists in applying Elias penultimate
  coding to the dictionary $\Delta(\mathbf{x})$ of a string $\mathbf{x}$ from $ \Np^n$ and
  then pattern coding to the pattern $\psi(\mathbf{x}).$ 

Then for any $\alpha$ such that $1 < \alpha\leq 5/2$, there exists a
constant $K$ depending on $\alpha$ and $C$ such that
\begin{displaymath}
  R^+(Q^n,\Lambda^n_{C\cdot^{-\alpha}})  \leq K n^{1/\alpha}
\log n 
\end{displaymath}

\end{thm}

\begin{proof}
For a given value of $C$ and $\alpha,$
the Elias encoding of the dictionary  uses on average
\begin{eqnarray*}
\E\left[|\Delta_n|\right] 
& = & K' n^{\frac{1}{\alpha}}\log n
\end{eqnarray*}
bits (as proved in Appendix \ref{Edelta}), for some constant $K'$
depending on $\alpha$ and $C$.

If our pattern coder reaches (approximately) the minimax pattern redundancy 
$$ R^+_\Psi\left(\Psi_{1:n}\right) = \inf_{q\in\mathfrak{M}_1\left(\Np^n\right)} \sup_{P\in\mathfrak{M}_1\left(\Np\right)} \E_P \left[\log \frac{P^{\otimes n}(\Psi_{1:n})}{q(\Psi_{1:n})}\right], $$ 
the encoding of the pattern uses on average 
\begin{eqnarray*}
H(\Psi_{1:n}) + R^+_\Psi\left(\Psi_{1:n}\right)&\leq & H(X_{1:n}) + R^+_\Psi\left(\Psi_{1:n}\right) \hbox{ bits.}
\end{eqnarray*}
But in \cite{MR2095850}, the authors show that $R^+_\Psi\left(\Psi_{1:n}\right)$
is upper-bounded by $O\left(\sqrt{n}\right)$ and even
$O\left(n^{\frac{2}{5}}\right)$ according to \cite{shamir:2004}
(actually, these bounds are even satisfied by the minimax individual
pattern redundancy). 
\end{proof}

This remarkably simple method is however expected to have a  poor
performance when $\alpha$ is large. Indeed, it is proved in \cite{garivier:2006} that
$R^+_\Psi\left(\Psi_{1:n}\right)$ is lower-bounded by
$1.84\left(\frac{n}{\log n}\right)^{\frac{1}{3}}$ (see also \cite{MR2234457} and references therein),
 which indicates that pattern coding is probably suboptimal as soon as $\alpha$ is
larger than $3$. 

\subsection{An approximately asymptotically adaptive censoring code}

Given the limited scope of the pattern coding method, we will attempt
to turn the censoring code into an adaptive method, that is to tune 
the cutoff sequence so as to model the source statistics. 
As the cutoffs are chosen in such a way that they model the
tail-heaviness of the source, we are facing a tail-heaviness
estimation problem . In order to focus on the most important issues 
we do not attempt to develop a sequential algorithm. 
The $n+1$th cutoff $K_{n+1}$ is chosen according to  the number of \emph{distinct}  symbols
$Z_n(\mathbf{x})$ in $\mathbf{x}$. 

This is a reasonable method if the probability mass function defining
the source statistics $P^1$
actually decays like  $\frac{1}{k^\alpha}$.
Unfortunately, sparse distributions consistent with
$\Lambda_{\cdot^{-\alpha}}$ 
may lead this project astray. 
If, for example, $\left(Y_n\right)_n$ is a sequence of geometrically distributed
 random variables, and if $X_n=\left\lfloor2^{\frac{Y_n}{\alpha}}\right\rfloor$,
 then the distribution of the $X_n$ just fits in $\Lambda_{C\cdot^{-\alpha}}$ 
but obviously $Z_n(X_{1:n}) = Z_n\left(Y_{1:n}\right) = O\left(\log n\right)$.

Thus, rather than attempting to handle
$\cup_{\alpha>0} \Lambda_{\cdot^{-\alpha}},$ we focus on subclasses
$\cup_{\alpha>0} \mathcal{W}_{\alpha}$, where 
\begin{displaymath}
\mathcal{W}_{\alpha} =
 \left\{  P~:~P\in \Lambda_{\cdot^{-\alpha}},~0<\lim\inf_k k^\alpha P^1(k)\leq \lim\sup_k k^\alpha P^1(k)<\infty \right\}\, . 
\end{displaymath}

The rationale for tuning cutoff $K_n$  using $Z_n$ comes from the
following two propositions. 

\begin{prop}\label{ECn}
For every memoryless source  $P \in\mathcal{W}_\alpha$, there exist constants $c_1$ and $c_2$ such that for all positive integer $n$, 
\begin{equation*}
c_1 n^{1/\alpha} \leq  \E[Z_n] \leq c_2  n^{1/\alpha}.
\end{equation*}
\end{prop}

\begin{prop}
The number of distinct symbols $Z_n$ output by a memoryless source
satisfies a Bernstein inequality:
\begin{equation}
  \label{eq:conc:Z_n}
  P \left\{ Z_n\leq \frac{1}{2}\EXP[Z_n]\right\} \leq \mathe^{- \frac{\EXP[Z_n]}{8}}
  \, . 
\end{equation}
\end{prop}
\begin{proof}
Note that $Z_n$  is a function of $n$ independent random
variables. Moreover, $Z_n$ is a configuration function as defined
defined by \cite{talagrand:1995} since $Z_n(\boldsymbol{x}) $ is the size of a maximum
subsequence of $\boldsymbol{x}$ satisfying  an hereditary property
(all its symbols are pairwise distinct). 
Using the main theorem in~\cite{boucheron:lugosi:massart:2000}, this is enough to conclude.
\end{proof}
Noting that $Z_n\geq 1$, we can derive the following inequality that will prove useful later on:
\begin{eqnarray}
\E\left[\frac{1}{Z_n^{\alpha-1}}\right] & = & \E\left[
  \frac{1}{Z_n^{\alpha-1}}\1_{Z_n > \frac{1}{2}\E[Z_n]} \right] +
\E\left[\frac{1}{Z_n^{\alpha-1}}\1_{Z_n\leq \frac{1}{2}\E[Z_n]}\right]
\notag \\
& \leq & \frac{1}{\left(\frac{1}{2}\E[Z_n]\right)^{\alpha-1}} + P\left(Z_n \leq \frac{1}{2}\E[Z_n]\right).\label{majexp}
\end{eqnarray}
We consider here a modified version of \texttt{CensoringCode} that
operates similarly, except that
\begin{enumerate}
\item the string $\mathbf{x}$ is first scanned completely to determine
  $Z_n\left(\mathbf{x}\right)$;
\item the constant cutoff
$\hat{K}_n= \mu Z_n$ 
is used for all symbols $\mathbf{x}_i$, $1\leq i \leq n$, where $\mu$ is some
positive constant.
\item the value of $K_n$ is encoded using Elias penultimate code and transmitted before \texttt{C1} and
  \texttt{C2}. 
\end{enumerate}
Note that this version of the algorithm is not sequential because of
the initial scanning.

\begin{algorithm}
\caption{AdaptiveCensoringCode}
\begin{algorithmic}
\label{algo:censoring:adpative}
\STATE $\text{\em cutoff} \leftarrow \mu\,Z_n(\mathbf{x})$ \COMMENT{Determination of the constant cutoff}
\STATE $\text{\em counts} \leftarrow [1/2, 1/2,\ldots ]$
\FOR{$i$ from $1$ to $n$}{
\IF{$x[i]\leq \text{\em cutoff}$}{
  \STATE 
$\text{ArithCode}(x[i],\text{\em     counts}[0:\text{\em  cutoff}]) $
}
\ELSE{
\STATE $\text{ArithCode}(0,\text{\em
    counts}[0:\text{\em  cutoff}]) $
\STATE \texttt{C1}$ \leftarrow $\texttt{C1}$\cdot \text{EliasCode}(x[i]) $
\STATE $\text{\em
    counts}[0] \leftarrow \text{\em
    counts}[0] +1$}
\ENDIF
\STATE $\text{\em
    counts}[x[i]] \leftarrow \text{\em
    counts}[x[i]] +1$
}
\ENDFOR
\STATE \texttt{C2}$ \leftarrow \text{ArithCode}()$
\STATE $\texttt{C}_1 \cdot \texttt{C}_2$
\end{algorithmic}
\end{algorithm}

  We may now assert.

\begin{thm}
The  algorithm \texttt{AdaptiveCensoringCode} is approximately
asymptotically adaptive with respect to $\bigcup_{\alpha>0}
\mathcal{W_\alpha}.$ 
\end{thm}


\begin{proof}
Let us again denote by \texttt{C1}$(\mathbf{x})$ and
\texttt{C2}$(\mathbf{x})$ the two parts of the code-string associated
with $\mathbf{x}.$

Let $\hat{L}$ be the codelength of the output of algorithm
\texttt{AdaptiveCensoringCode}.

For any source  $P$:
\begin{eqnarray*}
\E_P \left[\hat{L}\left(X_{1:n}\right)\right] -H(X_{1:n})& = &\E_{P}\Bigg[ \ell(\hat{K}_n ) +
  \left|\hbox{\texttt{C1}}\left(X_{1:n}\right)\right| + \left|\hbox{\texttt{C2}}\left(X_{1:n}\right)\right| \Bigg] - n\sum_{k=1}^{\infty} P^1(k)\log \frac{1}{P_1(k)}\\
&\leq & \E_P\left[\ell(\hat{K}_n )\right] +  \E_P\left[ \left|\hbox{\texttt{C1}}\left(X_{1:n}\right)\right|\right] + \E_P\left[\left|\hbox{\texttt{C2}}\left(X_{1:n}\right)\right| - n\sum_{k=1}^{\hat{K}_n} P^1(k)\log \frac{1}{P_1(k)}\right].
\end{eqnarray*}
As function $\ell$ is increasing and
equivalent to $\log $ at infinity, the first summand is obviously $o\left(\E_P\left[\ell(\hat{K}_n )\right]\right)$.
Moreover, if $P\in \mathcal{W}_\alpha$ there exists $C$ such that $P^1(k)\leq\frac{C}{k^\alpha}$ and the second summand satisfies:
\begin{eqnarray*}
\E_P\left[ \left|\hbox{\texttt{C1}}\left(X_{1:n}\right)\right|\right] & = & \E_P\left[\sum_{k\geq\hat{K}_n +1}P^1(k) \ell (k)\right]\\
&\leq & nC\E_P\left[\int_{\hat{K}_n}^\infty\frac{\ell\left(x\right)}{x^\alpha} \text{d}x\right]\\
&=& nC\E_P\left[\frac{1}{\hat{K}_n^{\alpha-1}}\int_{1}^\infty\frac{\ell \left(\hat{K}_nu\right)}{u^\alpha} \text{d}u\right]\\
& \leq & nC\E_P\left[\frac{1}{\hat{K}_n^{\alpha-1}}\right] \int_{1}^\infty\frac{\log\left(u\right)}{u^\alpha} \,\text{d}u \left(1+o(1)\right)
\\ & = & O\left(n^{\frac{1}{\alpha}}\log n\right)
\end{eqnarray*}
by Proposition (\ref{ECn}) and Inequality (\ref{majexp}).

By Theorem \ref{prop:finitealph}, every string  $x\in\Np^n$ satisfies 
\begin{eqnarray*}
\left|\hbox{\texttt{C2}}\left(x\right)\right| - n\sum_{k=1}^{\hat{K}_n} P^1(k)\log \frac{1}{P_1(k)} & \leq & \frac{\hat{K}_n}{2} \log n +2.
\end{eqnarray*}
Hence, the third summand is upper-bounded as:
\begin{eqnarray*}
\E_P\left[\left|\hbox{\texttt{C2}}\left(X_{1:n}\right)\right| - n\sum_{k=1}^{\hat{K}_n} P^1(k)\log \frac{1}{P_1(k)}\right] &\leq &  \frac{\E_P\left[\hat{K}_n\right]}{2} \log n +2\\
 & = & O\left(n^{\frac{1}{\alpha}}\log n\right)
\end{eqnarray*}
which finishes to prove the theorem.
\end{proof}

\bibliographystyle{abbrvnat}
\bibliography{bigbib}

\appendices

\section{Upper-bound on minimax regret}
This sections contains the proof of the last inequality in Theorem
\ref{prop:finitealph}.

The minimax regret is not larger than the maximum regret of the
Krichevsky-Trofimov mixture over $m$-ary alphabet over strings of
length $n.$ The latter is classically upper-bounded by
\begin{displaymath}
  \log
  \left(\frac{\Gamma(n+\frac{m}{2})\Gamma(\frac{1}{2})}{\Gamma(n+\frac{1}{2})\Gamma(\frac{m}{2})}\right) \, , 
\end{displaymath}
as proved for example in \citep{csiszar:1990}.

Now  the Stirling approximation to the Gamma function
\citep[See][ Chapter XII]{whittaker:watson} asserts
that for any $x>0,$  there exists $\beta\in[0,1]$ such that
\begin{displaymath}
  \Gamma(x) = x^{x-\frac{1}{2}} \text{e}^{-x} \sqrt{2\pi}
  \text{e}^{\frac{\beta}{12 x}} \, . 
\end{displaymath}
Hence,
\begin{eqnarray}
   \log
  \left(\frac{\Gamma(n+\frac{m}{2})\Gamma(\frac{1}{2})}{\Gamma(n+\frac{1}{2})\Gamma(\frac{m}{2})}\right) &= &
\left(n+\frac{m-1}{2}\right) \log \left( n+\frac{m}{2} \right) - n\log \left( n+\frac{1}{2}  \right)- \frac{m-1}{2}\log\frac{m}{2}\label{ligne1}\\
& & -\left( n+\frac{m}{2} \right) + n+\frac{1}{2} +\frac{m}{2}\label{ligne2}\\
& & -\log \sqrt{2\pi} + \log\sqrt{2\pi} + \log\sqrt{2\pi}-\log\sqrt{\pi}\label{ligne3}\\
& &+\frac{\beta_1}{12\left( n+\frac{m}{2} \right)} -\frac{\beta_2}{12\left( n+\frac{1}{2} \right)} -\frac{\beta_3}{6m}\label{ligne4}
\end{eqnarray}
for some $\beta_1,\beta_2,\beta_3\in[0,1].$
Now, (\ref{ligne2})+(\ref{ligne3})+(\ref{ligne4}) is smaller than $\frac{1}{2}+\log \sqrt{2}+\frac{1}{12\left( n+\frac{m}{2} \right)}\leq 2$, and (\ref{ligne1}) equals:
$$
\frac{m-1}{2}\log n
 +\left(  n\log \frac{n+\frac{m}{2}}{n+\frac{1}{2}}-\frac{m-1}{2}\log e \right)
 + \left( \frac{m-1}{2}\log\frac{n+\frac{m}{2}}{\frac{m}{2}} + \frac{m-1}{2}\log e - \frac{m-1}{2}\log n \right)
$$
But 
\begin{eqnarray*}
n\log \frac{n+\frac{m}{2}}{n+\frac{1}{2}} = n\log \left( 1+\frac{\frac{m-1}{2}}{n+\frac{1}{2}} \right)\leq n\frac{\frac{m-1}{2}}{n+\frac{1}{2}}\log e \leq \frac{m-1}{2}\log e, 
\end{eqnarray*}
and
\begin{eqnarray*}
\frac{m-1}{2}\log\frac{n+\frac{m}{2}}{\frac{m}{2}} + \frac{m-1}{2}\log e - \frac{m-1}{2}\log n & = &
\frac{m-1}{2}\log \frac{\left( n+\frac{m}{2}\right)e}{\frac{nm}{2}} \leq 0 
\end{eqnarray*}
if $\left( n+\frac{m}{2}\right)e\leq \frac{nm}{2}$, that is $\frac{2}{m}+\frac{1}{n}\leq \frac{1}{e}$, which is satisfied as soon as $m$ and $n$ are both at least equal to $9$. For the smaller values of $m,n\in\{2,\ldots, 8\}$ the result can be checked directly.


\section{Lower bound on redundancy for power-law envelopes}
\label{sec:lower:bound:envelope}
In this appendix we derive a lower-bound for power-law envelopes using
Theorem \ref{prop:desperate}. 
Let $\alpha$ denote a real larger than $1.$
Let $C$ be such that $C^{1/\alpha}>4$.
As the envelope function is defined by $f(i)=1\wedge C/i^\alpha,$ the
constant 
\begin{math}
    c(\infty) = \sum_{i\geq 1} f(2i) 
\end{math}
satisfies
\begin{displaymath}
\frac{\alpha}{\alpha-1}\frac{C^{1/\alpha}}{2}-1 \leq c(\infty)\leq  \frac{C^{1/\alpha}}{2} +
\frac{C}{(\alpha-1)2^\alpha}
\left(\frac{C^{1/\alpha}}{2}\right)^{1-\alpha} \, . 
\end{displaymath}
 The condition on $C$ and $\alpha$ warrants that, for sufficiently large
 $p$, we have $c(p)>1$ (this is indeed true for $p>C^{1/\alpha}$).

We choose $p=an^{\frac{1}{\alpha}}$ for $a$ small enough to have 
$$\frac{(1-\lambda)
  C\epsilon}{\left(2a\right)^{\frac{1}{\alpha}}c(\infty)}>10,$$
so that condition 
\begin{math}
  (1-\lambda)n\frac{f(2p)}{c(p)}>\frac{10}{\epsilon}
\end{math}
is satisfied for $n$ large enough.
Then
\begin{equation*}
  R^+(\Lambda^n_f) \geq C(p,n,\lambda, \epsilon) \sum_{i=1}^p\left(
  \frac{1}{2} \log \frac{n\left(1-\lambda\right)\pi f(2i)}{2c(p)e}-\epsilon\right),
\end{equation*} 
where 
$C(p,n,\lambda, \epsilon) = 
\frac{1}{1+\frac{\left(2a\right)^\alpha c(\infty)}{C\lambda^2}}\left(1-
    \frac{4}{\pi}\sqrt{\frac{5c(\infty)\left(2a\right)^\alpha}{\left(1-\lambda\right)C\epsilon }}\right)$,  and
\begin{eqnarray*}
\sum_{i=1}^p\left(
  \frac{1}{2} \log \frac{n\left(1-\lambda\right)\pi
  f(2i)}{2c(p)e}-\epsilon\right) & \leq & \frac{p}{2}\log n
  -\frac{\alpha}{2}\sum_{i=1}^{p} \log i + \left(\frac{1}{2}\log
  \frac{\left(1-\lambda\right)\pi C}{2^{1+\alpha}c(\infty)e} -\epsilon\right) p\\
&=& \frac{p}{2}\log n - \frac{\alpha}{2}\left(p \log p -p +o(p)\right)
  +  \left(\frac{1}{2}\log
  \frac{\left(1-\lambda\right)\pi C}{2^{1+\alpha}c(\infty)e} -\epsilon\right) p\\
&=& \frac{an^{\frac{1}{\alpha}}}{2}\log n - \frac{\alpha}{2}\left(
  an^{\frac{1}{\alpha}} \log a+\frac{a}{\alpha}n^{\frac{1}{\alpha}}
  \log n -an^{\frac{1}{\alpha}}
  +o\left(n^{\frac{1}{\alpha}}\right)\right)\\& &\hspace{1cm} 
  +  \left(\frac{1}{2}\log
  \frac{\left(1-\lambda\right)\pi C}{2^{1+\alpha}c(\infty)e} -\epsilon\right) an^{\frac{1}{\alpha}}\\
&=& \left(\frac{\alpha}{2}\left(1-\log a\right)+\frac{1}{2}\log
  \frac{\left(1-\lambda\right)\pi C}{2^{1+\alpha}c(\infty)e} -\epsilon+o(1)\right) an^{\frac{1}{\alpha}}.
\end{eqnarray*}
For $a$ small enough, this gives the existence of a positive constant
$\eta$ such that $R^+(\Lambda^n_f) \geq \eta n^{\frac{1}{\alpha}}$.

\section{Proof of Proposition~\ref{ECn}}

Suppose that there exist  $k_0$, $c$ and $C$ such that for all $k\geq
k_0,$ $\frac{c}{k^\alpha}\leq p_k\leq \frac{C}{k^\alpha}$.

For $0\leq x \leq \frac{1}{2}$, it holds that $ -(2\log 2) x \leq \log (1-x) \leq -x$ and thus
$$\text{e}^{-(2\log 2) nx}\leq (1-x)^n \leq \text{e}^{-nx}.$$

Hence (as $p_k\leq \frac{1}{2}$ for all $k\geq 2$) :
\begin{eqnarray*}
\sum_{k=k_0}^\infty \left(
  1-\left(1-\frac{c}{k^\alpha}\right)^n\right)&  \leq  \E[Z_n] & \leq
\sum_{k=1}^\infty\left( 1-\left(1-\frac{C}{k^\alpha}\right)^n\right)  \\
\sum_{k=k_0}^\infty \left( 1-e^{-\frac{cn}{k^\alpha}}\right) & \leq
\E[Z_n] &\leq  1+ \sum_{k=2}^\infty \left( 1-e^{-\frac{(2\log
      2)Cn}{k^\alpha}}\right) \\
\int_{k_0}^\infty \left( 1-e^{-\frac{cn}{x^\alpha}}\right) \text{d} x
&\leq \E[Z_n] &\leq  1+ \int_1^\infty  \left(1-e^{-\frac{(2\log
      2)Cn}{x^\alpha}}\right) \text{d}x. 
\end{eqnarray*}
But, for any $t,K>0$, it holds that
\begin{eqnarray*}
\int_t^\infty  \left(1-e^{-\frac{Kn}{x^\alpha}}\right)\text{d} x & = &
\frac{\left(Kn\right)^{1/\alpha}}{\alpha} \int_0^{\frac{Kn}{t^\alpha}}
\frac{1-e^{-u}}{u^{1+1/\alpha}} \text{d} u.
\end{eqnarray*}

Thus, by noting that integral
$$A(\alpha)= \int_0^\infty  \frac{1-e^{-u}}{u^{1+1/\alpha}} \text{d} u,$$
is finite, we get 
\begin{eqnarray*}
\frac{c^{1/\alpha}A(\alpha)}{\alpha} n^{1/\alpha}\left(1-o(1)\right)& \leq  \E[Z_n] & \leq \frac{\left((2\log 2)C\right)^{1/\alpha} A(\alpha)}{\alpha}  n^{1/\alpha}.
\end{eqnarray*}

\section{Expected size of  dictionary encoding}
\label{Edelta}
Assume  that the probability mass function $(p_k)$ satisfies  
$p_k\leq \frac{C}{k^\alpha}$ for $C>0$ and all $k\geq 0$.
Then, using  Elias penultimate code for the first occurrence of each
 symbol in $X_{1:n},$ 
the expected length of the binary encoding of the dictionary  can  be
upper-bounded in the following way. Let $U_k$ be equal to $1$ if
symbol  $k$ occurs in $X_{1:n}$, and equal to $0$ otherwise.
\begin{eqnarray*}
\E\left[\left|\Delta_n\right|\right] & = & \E\left[\sum_{k=1}^\infty {U_k} \ell(k)\right] 
\\
& = & \sum_{k=1}^\infty \E\left[ {U_k} \ell(k)\right] \\
& \leq & \sum_{k=1}^\infty\left( 1-\left(1-\frac{C}{k^\alpha}\right)^n\right) \ell(k)\\
& \leq & 2 \left( 1+ \sum_{k=2}^\infty \left(1-e^{-\frac{(2\log 2)Cn}{k^\alpha}}\right)\log k\right)\\
& \leq & 2 \left(1+ \int_1^\infty  \left(1-e^{-\frac{(2\log
        2)Cn}{x^\alpha}}\right) \log x\, \text{d} x \right)\\
& \leq & 2 \left(\frac{\left(\left(2\log 2\right)Cn\right)^{1/\alpha}}{\alpha^2}\int_0^{\left(2\log 2\right)Cn}  \frac{1-e^{-u}}{u^{1+1/\alpha}} \log\left(\frac{\left(2\log 2\right)Cn}{u}\right)\text{d} u\right)\\
& \leq & T\, \frac{\left(\left(2\log 2\right)Cn\right)^{1/\alpha}}{\alpha^2}\log n\int_0^{\infty}  \frac{1-e^{-u}}{u^{1+1/\alpha}}\text{d} u
\end{eqnarray*}
for some positive constant $T$.


\subsection*{Acknowledgment}
The authors wish to thank L\'azslo Gy\"orfi for stimulating  
discussion and helpful comments. 
\end{document}